\documentclass[12pt]{amsart}
\usepackage{amssymb}
\usepackage{tabularx}   
\usepackage{mathtools}  

\usepackage{enumerate}             

\makeatletter
\@namedef{subjclassname@2010}{%
  \textup{2010} Mathematics Subject Classification}
\makeatother

\newtheorem{theorem}{Theorem}[section]
\newtheorem{intro}{Theorem}
\newenvironment{introtheorem}[1]
  {\renewcommand\theintro{#1}\intro}
  {\endintro}
\newtheorem{proposition}[theorem]{Proposition}
\newtheorem{corollary}[theorem]{Corollary}
\newtheorem{lemma}[theorem]{Lemma}

\newtheorem*{theorem*}{Theorem}
\newtheorem*{corollary*}{Corollary}
\newtheorem*{claim}{Claim}

\theoremstyle{definition}
\newtheorem{remark}{Remark}

\newtheorem*{notation}{Notation}

\newtheorem*{example}{Example}
\newtheorem*{examples}{Examples}

\numberwithin{equation}{section}

\newcommand{\tO}{\mathtt 0}                       
\newcommand{\tL}{\mathtt 1}                       
\newcommand{\tN}{\mathtt q}                       

\newcommand{\dN}{\mathbb N}                         
\newcommand{\dZ}{\mathbb Z}                         
\newcommand{\dC}{\mathbb C}                         
\newcommand{\abs}[1]{\left\lvert#1\right\rvert}     
\newcommand{\dens}{\operatorname{dens}}             
\DeclareMathOperator{\LandauO}{\mathcal{O}}         
\allowdisplaybreaks     
\begin{document}

\baselineskip=17pt

\title[Divisibility of binomial coefficients]%
{An explicit generating function arising in counting binomial coefficients divisible by powers of primes}
\author[L. Spiegelhofer]{Lukas Spiegelhofer}
\address{Institute of Discrete Mathematics and Geometry,
Vienna University of Technology,
Wiedner Hauptstrasse 8--10, 1040 Vienna, Austria}
\email{lukas.spiegelhofer@tuwien.ac.at}
\author[M. Wallner]{Michael Wallner}
\email{michael.wallner@tuwien.ac.at}
\date{}
\thanks{The first author acknowledges support by the project MuDeRa (Multiplicativity, Determinism, and Randomness),
which is a joint project between the
ANR (Agence Nationale de la Recherche) and the FWF (Austrian Science Fund),
and also by project F5502-N26 (FWF),
which is a part of the Special Research Program
``Quasi Monte Carlo Methods: Theory and Applications''.
The second was supported by Project SFB F50-03 (FWF),
which is a part of the Special Research Program ``Algorithmic and Enumerative Combinatorics''.}
\begin{abstract}
For a prime $p$ and nonnegative integers $j$ and $n$ let $\vartheta_p(j,n)$ be the number of entries in the $n$-th row of Pascal's triangle that are exactly divisible by $p^j$.
Moreover, for a finite sequence $w=w_{r-1}\cdots w_0\neq 0\cdots 0$ in $\{0,\ldots,p-1\}$ we denote by $\lvert n\rvert_w$ the number of times that $w$ appears as a factor (contiguous subsequence) of the base-$p$ expansion
$n_{\mu-1}\cdots n_0$ of $n$.
It follows from the work of Barat and Grabner (\textit{Distribution of binomial coefficients and digital functions}, J. London Math. Soc. (2) 64(3), 2001),
that $\vartheta_p(j,n)/\vartheta_p(0,n)$ is given by a polynomial $P_j$ in the variables $X_w$,
where $w$ are certain finite words in $\{0,\ldots,p-1\}$,
and each variable $X_w$ is set to $\lvert n\rvert_w$.
This was later made explicit by Rowland (\textit{The number of nonzero binomial coefficients modulo $p^\alpha$}, J. Comb. Number Theory 3(1), 2011), independently from Barat and Grabner's work,
and Rowland described and implemented an algorithm computing these polynomials $P_j$.
In this paper, we express the coefficients of $P_j$ using generating functions,
and we prove that these generating functions can be determined explicitly by means of a recurrence relation.
Moreover, we prove that $P_j$ is uniquely determined,
and we note that the proof of our main theorem also provides a new proof of its existence.
Besides providing insight into the structure of the polynomials $P_j$,
our results allow us to compute them in a very efficient way.
\end{abstract}

\maketitle
\section{Introduction}
The history of binomial coefficients in congruence classes modulo $m$ begins not later than in the middle of the 19th century,
when Kummer~\cite{K1852} stated his famous theorem on the highest prime power $p^m$ dividing a binomial coefficient $\binom nt$:
$m$ is the number of \emph{borrows} occurring in the subtraction $n-t$ in base $p$.
In other words, this is the number of indices $k$ such that $n\bmod p^k<t\bmod p^k$.
Kummer's theorem was generalised to multinomial and $q$-multinomial coefficients by Fray~\cite{F1967}, and to generalised binomial coefficients by Knuth and Wilf~\cite{KW1989}.

A complete list of results related to Pascal's triangle modulo powers of primes would go beyond the scope of any research paper;
we refer the reader to the surveys~\cite{G1997,Si1980} by Granville and Singmaster, respectively, for an overview of the topic.
The question also attracts other areas of research:
in~\cite[Section~14.6]{AS2003} and~\cite{AB1997}, connections with automatic sequences and combinatorics on words are highlighted.
Moreover, the paper~\cite{AS2008} considers the related question of counting coefficients equal to a given value of a polynomial over a finite field.

In this paper we restrict ourselves to questions concerning \emph{exact divisibility} of binomial coefficients by powers of primes.
This means that we are only concerned with the residue class $p^j$ modulo $p^{j+1}$,
in other words, we study the case $\nu_p\binom nt=j$,
where $\nu_p(m)$ denotes the largest $k$ such that $p^k\mid m$.

We therefore introduce the following notion, which is central in our paper.
Let $j$ and $n$ be nonnegative integers and $p$ a prime number, and define
\begin{equation*}
\vartheta_p(j,n)
=\left\lvert\left\{t\in\{0,\ldots,n\}:\nu_p\binom nt=j\right\}\right\rvert.
\end{equation*}
Put into words,
$\vartheta_p(j,n)$ is the number of entries in the $n$-th row of Pascal's triangle that are exactly divisible by $p^j$.
The case $j=0$ can be reduced to properties of the base-$p$ expansion of the row number $n$ by appealing to Lucas' congruence~\cite{L1878}.
This well-known congruence asserts that for $t\leq n$ having the (not necessarily proper) base-$p$ representations $n=(n_{\mu-1}\cdots n_0)_p$ and $t=(t_{\mu-1}\cdots t_0)_p$, we have
\[\binom nt\equiv\binom{n_{\mu-1}}{t_{\mu-1}}\cdots \binom{n_0}{t_0}\bmod p.\]
Since $p$ is a prime number,
we have $p\nmid\binom nt$ if and only if none of the factors is divisible by $p$,
which in turn is equivalent to $t_i\leq n_i$ for all $i<\mu$.
We obtain,
denoting by $\abs{n}_a$ the number of times the digit $a\neq 0$ occurs in the base-$p$ expansion of $n$,
\begin{equation*}
\vartheta_2(0,n)=2^{\abs{n}_1}
\end{equation*}
for the case $p=2$ (Glaisher~\cite{G1899}) and more generally (Fine~\cite{F1947})
\begin{equation}\label{eqn:fine}
\vartheta_p(0,n)=
\prod_{0\leq i<\mu}(n_i+1)
=2^{\abs{n}_1}3^{\abs{n}_2}4^{\abs{n}_3}\cdots p^{\abs{n}_{p-1}}.
\end{equation}
Lucas' congruence has been generalised and extended in different directions,
see for example~\cite{F1967}, \cite{K1968} (re-proved in~\cite{S1974a}), ~\cite{DW1990,G1992,G1997}; moreover~\cite{D1919} for an account of less recent results.
In order to be able to formulate our results concerning general $j\geq 0$,
we need some notation.
\begin{notation}
The letter $p$ always denotes a prime number;
we use typewriter font to indicate digits in the base-$p$ expansion,
except for variables representing digits.
For the $(p-1)$-st digit we write $\tN$,
a letter supposed to be a mnemonic relating to $\mathtt 9$ in the decimal expansion.
If $v$ is an infinite word over the alphabet $\{\tO,\ldots,\tN\}$
such that $v_i\neq \tO$ for only finitely many $i\geq 0$,
let $(v)_p=\sum_{i\geq 0} v_ip^i$ be the integer represented by $v$ in base $p$.
Moreover, if $w=w_{\mu-1}\cdots w_0\in\{\tO,\ldots,\tN\}^\mu$ contains at least one nonzero digit and $v$ is as above,
let $\abs{v}_w$ be the number of times that $w$ occurs as a factor of $v$.
More precisely,
\begin{equation*}
\abs{v}_w=\abs{\{i\geq 0:v_{i+\mu-1}\cdots v_i=w_{\mu-1}\cdots w_0\}}.
\end{equation*}
For finite words $v$ we extend the above notions by padding with zeros.
Moreover, if $n$ is a nonnegative integer and $n=(v)_p$,
we set $\abs{n}_w\coloneqq \abs{v}_w$.
Occurrences of factors may overlap:
for example, for $p=2$ we have $\abs{42}_{\tL\tO\tL\tO}
=\abs{\tL\tO\tL\tO\tL\tO}_{\tL\tO\tL\tO}=2$.
Moreover, as a consequence of the padding with zeros 
we have $\abs{\tL}_{\tL}=\abs{\tL}_{\tO\tL}=\abs{\tL}_{\tO\tO\tL}=\cdots=1$,
while $\abs{\tL}_{\tL\tO}=0$.
\end{notation}
The following statement is an easy reformulation of~\cite[Theorem~2]{R2011}.
The method used for proving this theorem is very similar to the method used in the older paper~\cite[Theorem~5]{BG2001},
which proves a less detailed form of the result,
but can be adapted to yield the full statement.
See also Remark~\ref{rem:RBG}.
\begin{introtheorem}{0}[Rowland~\cite{R2011}--Barat--Grabner~\cite{BG2001}]\label{thm:rowland}
Let $p$ be a prime and $j\geq 0$.
Then $\vartheta_p(j,n)/\vartheta_p(0,n)$ is given by a polynomial $P_j$ of degree $j$ in the variables $X_w$, where $w$ ranges over the set
\begin{equation}\label{eqn:Wj_def}
W_j=\bigl\{w\in\{\tO,\ldots,\tN\}^\mu:2\leq\mu\leq j+1,w_{\mu-1}\neq \tO,w_0\neq \tN\bigr\},
\end{equation}
and $X_w$ is set to $\abs{n}_w$.
\end{introtheorem}
Note that $W_0=\emptyset$ and $P_0(x)=1$.
Determining $\vartheta_p(j,n)/\vartheta_p(0,n)$ by means of this theorem is a two-step procedure:
\begin{equation}\label{eqn:theta_rep}
n\mapsto \bigl(\abs{n}_w\bigr)_{w\in W_j}\mapsto P_j\Bigl(\bigl(\abs{n}_w\bigr)_{w\in W_j}\Bigr)
=\frac{\vartheta_p(j,n)}{\vartheta_p(0,n)}.
\end{equation}
Barat and Grabner~\cite[Theorem~5]{BG2001} used a representation of $\vartheta_p(j,n)/\vartheta_p(0,n)$
of this kind in order to establish an asymptotic formula for the partial sums $\sum_{0\leq n<N}\vartheta_p(j,n)$.
Their Theorem~5 generalises the case $j=0$~\cite{FGKPT1994} (see also~\cite{BG1996,S1977}),
and yields a quantitative version of the statement ``any integer divides almost all binomial coefficients''~\cite{S1974c}.

Theorem~\ref{thm:rowland} implies, as noted by Rowland, that
$n\mapsto \vartheta_p(j,n)/\vartheta_p(0,n)$ is a $p$-\emph{regular sequence}
in the sense of Allouche and Shallit~\cite{AS1992,AS2003}.
We will however not follow this line of research in this paper.

In Proposition~\ref{prp:unique_polynomial} we will prove that a polynomial $P_j$ as in Theorem~\ref{thm:rowland} is uniquely determined,
so that we may talk about the coefficients of $P_j$ without ambiguity.
These polynomials are the main object of study in this paper,
and we want to obtain a better understanding of its coefficients.
Our main theorem (restated in Section~\ref{sec:main}) concerns the behaviour of the coefficients of a single monomial in the sequence $(P_j)_{j\geq 0}$ of polynomials.
\begin{theorem*}
Let $W$ be the set of all words
$w_{\mu-1}\cdots w_0\in\{\tO,\ldots,\tN\}^\mu$
such that $\mu\geq 2$, $w_{\mu-1}\neq \tO$ and $w_0\neq \tN$.
Assume that $w^{(1)},\ldots,w^{(\ell)}\in W$,
and $k_1,\ldots,k_{\ell}$ are positive integers.
Let $c_j$ be the coefficient of the monomial
\[X_{w^{(1)}}^{k_1}\cdots X_{w^{(\ell)}}^{k_{\ell}}\]
in the polynomial $P_j$.
Then
\[
\sum_{j\geq 0}c_jx^j =
\frac 1{k_1!}\bigl(\log r_{w^{(1)}}(x)\bigr)^{k_1}\cdots
\frac 1{k_{\ell}!}
\bigl(\log r_{w^{(\ell)}}(x) \bigr)^{k_{\ell}},
\]
where $r_w$ is a rational function defined at $0$ such that $r_w(0)=1$.
\end{theorem*}
The rational function $r_w$ can be determined explicitly by means of a recurrence,
see Section~\ref{sec:main}.
The easiest nontrivial example is $r_{\tL\tO}(x)=1+x/2$ ($p=2$).
Note that the coefficients $c_j$ always belong to a fixed monomial $X_{w^{(1)}}^{k_1}\cdots X_{w^{(\ell)}}^{k_{\ell}}$.
However, in order to increase readability we will not emphasize this relationship by additional sub- or superscripts.
It will always be clear from the context which monomial is referred to.

As a direct consequence of our results we will obtain the following corollary.
\begin{corollary*}
Let $p=2$.
The coefficient $c_j$ of the monomial $X_{\tL\tO}$ in $P_j$ equals
$\left[x^j\right]\log(1+x/2)$.
In particular,
\[\sum_{j\geq 0}c_j=\log(3/2).\]
\end{corollary*}
This special case confirms an observation by Rowland~\cite{R2011},
who noted that a plot of the first few partial sums
$c'_j=c_0+\cdots+c_{j-1}$ ``suggests that the limit of this sequence exists''.
He computed the first seven polynomials
\[ P'_j=P_0+\cdots+P_{j-1} \]
with the help of his Mathematica package {\sc BinomialCoefficients},
which is based on his paper~\cite{R2011} and available from his website,
and determined the coefficients $c'_j$ that way.
By the above corollary the limit does exist indeed,
and its value is $\log(3/2)$.
It is however not true for each monomial $M$ that the sequence of coefficients of $M$ in $P'_j$ converges as $j\rightarrow \infty$,
nor is it the case that all coefficients of $P'_j$ are nonnegative.
A simultaneous counterexample for both questions is given by $X_{\tL\tO\tL\tO}$ (see the examples after Corollary~\ref{cor:behaviours}).
The sequence of coefficients of this monomial has the generating function
\[  \log\Bigl(1+\tfrac 12x^3/\bigl(1+x/2\bigr)^2\Bigr),  \]
which has a unique dominant singularity $x_0\sim-0.86408$.
Therefore negative signs occur infinitely often and the sequence of coefficients diverges to $\infty$ in absolute value (this is true for the coefficients in $P_j$ as well as in $P'_j$).

While the above results concern the behaviour of a single monomial in different polynomials $P_j$,
we will also prove an ``orthogonal'' result,
namely an asymptotic estimate of the number of nonzero coefficients in $P_j$ and $P'_j$ (Corollary~\ref{cor:number_of_terms}).

The results that we have outlined above provide answers to questions posed by Rowland~\cite{R2011} at the end of his paper.
For more details, we refer to Section~\ref{sec:main}.
Finally, we want to note that our main theorem together with the recurrence for $r_w$ enables us to compute the polynomials $P_j$ very efficiently (see Remark~\ref{rem:number_of_terms}).

We will also use the following notations in this article.
The integer $s_2(n)\coloneqq\abs{n}_\tL$ is the \emph{sum of digits} of $n$ in base $2$,
more generally $s_p(n)\coloneqq\abs{n}_\tL+2\abs{n}_{\mathtt 2}+\cdots+(p-1)\abs{n}_{\tN}$ is the sum of digits of $n$ in base $p$.
For a finite word $w$ we denote by $\abs{w}$ the length of $w$.
Finally, $\dN$ denotes the set of nonnegative integers.

\smallskip\noindent
\textit{Plan of the paper.}
In Section~\ref{sec:rec} we will meet the fundamental recurrence relation for the values $\vartheta_p(j,n)$, found by Carlitz~\cite{C1967},
while in Section~\ref{sec:higherpowers} we list some of the polynomials $P_j$ for the case $p=2$.
In Sections~\ref{sec:computing} and~\ref{sec:asymp}, we will state in detail the results we announced above, and study the rational functions $r_w$ more carefully.
Section~\ref{sec:simplified} gives an alternative form of the fundamental recurrence relation for $\vartheta_p(j,n)$,
which can be written as an elegant but enigmatic infinite product.
This also yields a new proof of Carlitz' recurrence relation.
Finally, we note in Section~\ref{sec:columns} that we can reuse the polynomials $P_j$ for the columns in Pascal's triangle.
Proofs not given in the main section are stated in Section~\ref{sec:proofs}.
\subsection{A recurrence for the values $\vartheta_p(j,n)$, and the case $j=1$}
\label{sec:rec}
Carlitz~\cite{C1967} gave a recurrence relation for the values $\vartheta_p(j,n)$,
which also involves another family $\psi_p$ defined by%
\footnote{Our notation differs slightly from Carlitz' who wrote $\theta_j(n)$ instead of $\vartheta_p(j,n)$ and $\psi_j(n)$ instead of $\psi_p(j,n)$, omitting $p$ altogether.}
\begin{equation*}
\psi_p(j,n)
=\abs{\left\{t\in\{0,\ldots,n\}:\nu_p\binom nt=j-\nu_p(n+1)\right\}}.
\end{equation*}
He then obtains \cite[Equations~(1.7)--(1.9)]{C1967}
for $n\geq 0$ and $j\geq 1$, using the convention $\psi_p(j,-1)=0$,
\begin{equation}\label{eqn:carlitz_rec}
\begin{aligned}
\vartheta_p(j,pn+a)&=(a+1)\vartheta_p(j,n)\\
&+(p-a-1)\psi_p(j-1,n-1),&0\leq a<p;\\
\psi_p(j,pn+a)&=(a+1)\vartheta_p(j,n)\\
&+(p-a-1)\psi_p(j-1,n-1),&0\leq a<p-1;\\
\psi_p(j,pn+p-1)&=p\psi_p(j-1,n).
\end{aligned}
\end{equation}
Rewriting the recurrence~\eqref{eqn:carlitz_rec} using the obvious identity
\[\psi_p(j,n)=\begin{cases}\vartheta_p(j-\nu_p(n+1),n),&j\geq \nu_p(n+1);\\
0,&j<\nu_p(n+1),\end{cases}\]
we obtain for $0\leq a<p$
\begin{multline}\label{eqn:theta_rec}
\vartheta_p(j,pn+a)=(a+1)\vartheta_p(j,n)\\+
\begin{cases}
(p-a-1)\vartheta_p(j-1-\nu_p(n),n-1),&j>\nu_p(n);\\
0,&j\leq \nu_p(n).\end{cases}
\end{multline}
Among other things, Carlitz evaluates $\vartheta_p(j,n)$ for special values of $n$,
using associated generating functions.
Moreover, he proves the explicit formula~\cite[Equation~(2.5)]{C1967},
saying that for the base-$p$ expansion $n=\sum_{i=0}^{\mu-1}n_ip^i$ we have
\begin{equation*}
\vartheta_p(1,n)=\sum_{0\leq i<\mu-1}
(n_{\mu-1}+1)\cdots(n_{i+2}+1)n_{i+1}(p-n_i-1)(n_{i-1}+1)\cdots(n_0+1).
\end{equation*}
By~\eqref{eqn:fine} this implies that
\begin{equation*}
\frac{\vartheta_p(1,n)}{\vartheta_p(0,n)}
=
\sum_{0\leq i<\mu-1}
\frac{n_{i+1}}{n_{i+1}+1}\cdot\frac{p-n_i-1}{n_i+1}.
\end{equation*}
In particular, counting identical summands, we obtain
\begin{equation}\label{eqn:carlitz_alt2}
\frac{\vartheta_p(1,n)}{\vartheta_p(0,n)}
=
\sum_{\substack{0\leq c,a<p\\c\neq 0,a\neq p-1}}
\frac c{c+1}\cdot \frac{p-a-1}{a+1}\abs{n}_{ca}.
\end{equation}
Note that we defined the quantity $\abs{n}_{ca}$ as the number of occurrences of $ca=n_{i+1}n_i$ in the base-$p$ expansion $n=\sum_{i=0}^\infty n_ip^i$.
Since $c$ is nonzero, this is equal to the number of occurrences of this pattern for $0\leq i<\mu-1$.
For the prime $p=2$ only one summand remains,
yielding the formula
\begin{equation*}
\frac{\vartheta_2(1,n)}{\vartheta_2(0,n)}=\frac 12\abs{n}_{\tL\tO}.
\end{equation*}
This formula was observed by Howard~\cite[Equation~(2.4)]{H1971},
see also~\cite[Theorem~2.2]{H1970}.
(The latter is however not correct if $n$ is a power of $2$.)
\subsection{The polynomials $P_j$ for $j>1$}\label{sec:higherpowers}
In 1971, Howard~\cite{H1971} also found formulas for
$\vartheta_2(2,n)$, $\vartheta_2(3,n)$, and $\vartheta_2(4,n)$
in terms of factor counting functions $\abs{n}_w$.
In different notation, he obtained the formulas
\begin{align*}
\frac{\vartheta_2(2,n)}{\vartheta_2(0,n)}&=
-\frac 18\abs{n}_{\tL\tO}
+\frac 18\abs{n}_{\tL\tO}^2
+\abs{n}_{\tL\tO\tO}
+\frac 14\abs{n}_{\tL\tL\tO},\\[2mm]
\frac{\vartheta_2(3,n)}{\vartheta_2(0,n)}&=
\frac 1{24}\abs{n}_{\tL\tO}
-\frac 1{16}\abs{n}_{\tL\tO}^2
-\frac 12\abs{n}_{\tL\tO\tO}
-\frac 18\abs{n}_{\tL\tL\tO}
+\frac 1{48}\abs{n}_{\tL\tO}^3
+\frac 12\abs{n}_{\tL\tO}\abs{n}_{\tL\tO\tO}\\
&+\frac 18\abs{n}_{\tL\tO}\abs{n}_{\tL\tL\tO}
+2\abs{n}_{\tL\tO\tO\tO}
+\frac 12\abs{n}_{\tL\tO\tL\tO}
+\frac 12\abs{n}_{\tL\tL\tO\tO}
+\frac 18\abs{n}_{\tL\tL\tL\tO},\\[2mm]
\frac{\vartheta_2(4,n)}{\vartheta_2(0,n)}&=
-\frac 1{64}\abs{n}_{\tL\tO}
+\frac{11}{384}\abs{n}_{\tL\tO}^2
-\frac 14\abs{n}_{\tL\tO\tO}
+\frac 1{32}\abs{n}_{\tL\tL\tO}
-\frac 1{64}\abs{n}_{\tL\tO}^3\\
&-\frac 38\abs{n}_{\tL\tO}\abs{n}_{\tL\tO\tO}
-\frac 3{32}\abs{n}_{\tL\tO}\abs{n}_{\tL\tL\tO}
-\abs{n}_{\tL\tO\tO\tO}
-\frac 12\abs{n}_{\tL\tO\tL\tO}
-\frac 12\abs{n}_{\tL\tL\tO\tO}\\
&-\frac 1{16}\abs{n}_{\tL\tL\tL\tO}
+\frac 1{384}\abs{n}_{\tL\tO}^4
+\frac 18\abs{n}_{\tL\tO}^2\abs{n}_{\tL\tO\tO}
+\frac 1{32}\abs{n}_{\tL\tO}^2\abs{n}_{\tL\tL\tO}
+\frac 12\abs{n}_{\tL\tO\tO}^2\\
&+\frac 14\abs{n}_{\tL\tO\tO}\abs{n}_{\tL\tL\tO}
+\frac 1{32}\abs{n}_{\tL\tL\tO}^2
+\abs{n}_{\tL\tO}\abs{n}_{\tL\tO\tO\tO}
+\frac 14\abs{n}_{\tL\tO}\abs{n}_{\tL\tO\tL\tO}\\
&+\frac 14\abs{n}_{\tL\tO}\abs{n}_{\tL\tL\tO\tO}
+\frac 1{16}\abs{n}_{\tL\tO}\abs{n}_{\tL\tL\tL\tO}
+4\abs{n}_{\tL\tO\tO\tO\tO}
+\abs{n}_{\tL\tO\tO\tL\tO}
+\abs{n}_{\tL\tO\tL\tO\tO}\\
&+\frac 14\abs{n}_{\tL\tO\tL\tL\tO}
+\abs{n}_{\tL\tL\tO\tO\tO}
+\frac 14\abs{n}_{\tL\tL\tO\tL\tO}
+\frac 14 \abs{n}_{\tL\tL\tL\tO\tO}
+\frac 1{16}\abs{n}_{\tL\tL\tL\tL\tO}.
\end{align*}
Moreover, Howard~\cite{H1973} found an expression for $\vartheta_p(2,n)$ for general primes $p$; see also~\cite{HSW1997b,W1990}.
We also refer to Spearman and Williams~\cite[Theorem~1]{SW1999}.
They re-proved the formulas above by expressing $\vartheta_2(j,n)/\vartheta_2(0,n)$
as a sum of nonoverlapping subwords of the binary expansion of $n$.
We note that the factors that are counted in the expressions for $\vartheta_2(j,n)$ always start with the digit $\tL$ (read from left to right) and end with the digit $\tO$.

That is, the words $w$ occurring in these expressions belong to the set $W_j$ defined in Theorem~\ref{thm:rowland}, for some $j\geq 1$.
By this theorem we can always require the condition $w\in W_j$,
while Proposition~\ref{prp:unique_polynomial} ensures uniqueness of an expression for~$\vartheta_2(j,n)$ as above.

We refrained from listing formulas for $j\geq 5$ for the obvious reason: $P_5$ contains $69$ monomials, $P_6$ already $174$.
\begin{remark}\label{rem:RBG}
As we noted before, the statement of the Theorem~\ref{thm:rowland} formulated by Rowland can already be found implicitly in Barat and Grabner~\cite{BG2001}.
That is, their method of proof can be adapted to show the theorem.
More precisely, in the course of proving Theorem~5 in that paper,
they proved that $\vartheta_p(j,n)/\vartheta_p(0,n)$ is a sum of products of block-additive functions.
Here a function $f:\dN\rightarrow\dC$ is called $\ell$-\emph{block-additive} in base $p$,
if there is a function $F:\{\tO,\ldots,\tN\}^\ell\rightarrow\dC$ satisfying $F(0,\ldots,0)=0$ such that for the base-$p$ expansion $n=\sum_{i\geq 0}\varepsilon_i p^i$ we have
\[f(n)=\sum_{i\geq 0}F(\varepsilon_{i+\ell-1},\ldots,\varepsilon_i).\]
These functions were first defined by Cateland in his thesis~\cite{C1992}.
We note that $\ell$-block-additive functions are precisely the complex linear combinations of factor counting functions $\abs{\cdot}_w$, where $w$ contains a nonzero letter and the length $\abs{w}$ is bounded by~$\ell$.
It follows from~\cite[(3.3),~(3.4)]{BG2001}
that the $\ell$-block-additive functions occurring in the representation of $\vartheta_p(j,n)/\vartheta_p(0,n)$ take only those factors $(w_{\mu-1}\cdots w_0)\in\{\tO,\ldots,\tN\}^\mu$ into account such that $w_{\mu-1}\neq\tO$ and $w_0\neq \tN$.
Moreover, enhancing the induction hypothesis in the proof of~\cite[Theorem~5]{BG2001},
it can be shown that only $\ell$-block-additive functions, where $1\leq\ell\leq j$, appear,
and that the occurring products of block-additive functions have length $\leq j$.
\end{remark}
Rowland~\cite{R2011} used an approach very similar to Barat and Grabner's~\cite{BG2001} (see also Spearman and Williams~\cite{SW1999}) in order to obtain Theorem~\ref{thm:rowland}.
More precisely, it follows from the proof of this theorem 
that the monomials $X_{w^{(1)}}\cdots X_{w^{(\ell)}}$ occurring in the polynomial $P_j$ satisfy
\begin{equation}\label{eqn:RBG_implicit}
\bigl\lvert w^{(1)}\bigr\rvert+\cdots +\bigl\lvert w^{(\ell)}\bigr\rvert-\ell\leq j.
\end{equation}
For example, if $p=2$ and $j=2$,
only the monomials $1$, $X_{\tL\tO}$, $X_{\tL\tO}^2$, $X_{\tL\tO\tO}$ and $X_{\tL\tL\tO}$ can occur.
Based on~\eqref{eqn:RBG_implicit} we will derive in Corollary~\ref{cor:number_of_terms} an upper bound for the number of monomials in $P_j$.

We note that we always write words from right to left,
since our interest in them stems from base-$p$ expansions of an integer.
Nevertheless, we call a \emph{prefix} of a word a contiguous subword containing the leftmost letter, while a \emph{suffix} is a contiguous subword containing the rightmost letter.
\section{Results}\label{sec:main}
\subsection{Computing the coefficients of $P_j$}\label{sec:computing}
Let $p$ be a prime number throughout this section.
For brevity of notation, we omit the index $p$ whenever there is no risk of confusion.
As in Theorem~\ref{thm:rowland}, let
\[    W_j=\bigl\{w\in\{\tO,\ldots,\tN\}^\mu:2\leq\mu\leq j+1,w_{\mu-1}\neq \tO,w_0\neq \tN\bigr\},    \]
moreover we define the set of \emph{admissible} words,
\[    W=\bigcup_{j\geq 1}W_j.    \]
We also define
\begin{align*}\widetilde W_j&=\bigl\{w\in\{\tO,\ldots,\tN\}^\mu:1\leq\mu\leq j+1,w_{\mu-1}\neq \tO\bigr\},\\
\widetilde W&=\bigcup_{j\geq 0}\widetilde W_j.    \end{align*}

In order to get meaningful statements on the coefficients of $P_j$,
we have to show that the polynomial $P_j$ is well-defined, i.e., uniquely determined.
Note that it is not clear a priori that there is only one polynomial $P_j$ representing $\vartheta_p(j,n)/\vartheta_p(0,n)$ as in~\eqref{eqn:theta_rep}:
the values inserted into this polynomial are not independent of each other,
therefore we can not use Lagrange interpolation directly for establishing uniqueness.
For example, we have $\abs{n}_{\tL\tO}\geq \abs{n}_{\tL\tO\tO}$ for all $n$,
so that not all tuples $(n_w)_{w\in W_j}$ of nonnegative integers can occur as family
$(\abs{n}_w)_{w\in W_j}$ of block counts of a nonnegative integer $n$.
Moreover, for the polynomial to be unique it is necessary that the blocks we are counting satisfy some restrictions,
since there are obvious identities such as $\abs{n}_{\tL}=\abs{n}_{\tO\tL}+\abs{n}_{\tL\tL}$.
We will show that the restriction
$w_{\mu-1}\neq \tO,w_0\neq \tN$ leads to a unique polynomial $P_j$ after all.
\begin{proposition}\label{prp:unique_polynomial}
There is at most one polynomial $P_j$ in the variables $X_w$,
where $w\in W$, such that
\[\frac{\vartheta_p(j,n)}{\vartheta_p(0,n)}
=P_j\left(\left(\abs{n}_w\right)_{w\in W}\right)\]
for all $n\geq 0$.
\end{proposition}

In order to prepare for the main theorem,
we define generating functions of the values $\vartheta_p(j,n)$,
which occupy 
a central position in the statements of the main results:
\begin{equation}\label{eqn:T_def}
T_n(x) \coloneqq \sum_{j\geq 0}\vartheta_p(j,n)x^j
=\sum_{0\leq t\leq n}x^{\nu_p\binom nt} .
\end{equation}
We note that the polynomials $T_n$ are studied in the recent paper~\cite{R2017} by Rowland, where it is shown that the sequence $(T_n)_{n\geq 0}$ of polynomials is a $p$-regular sequence.
Obviously, $T_n(x)$ is a polynomial of degree $\max_{0\le t\leq n}\nu_p\binom nt$, which is sequence~\texttt{A119387} in Sloane's OEIS~\cite{Sloane} for the case $p=2$.
The recurrence~\eqref{eqn:theta_rec} for $\vartheta_p$ translates to the generating functions $T_n(x)$ as follows:
\begin{equation}\label{eqn:T_rec}
\begin{aligned}
T_a(x)&=a+1,\\
T_{pn+a}(x)&=(a+1)T_n(x)+(p-a-1)x^{\nu_p(n)+1}T_{n-1}(x),
\end{aligned}
\end{equation}
for $n\geq 1$ and $0\leq a<p$.
We note the special case
\begin{align*}
T_{cp^t-\tL}(x)=T_{(c-\tL)\tN^t}(x)=cp^t,&&&1\leq c<p,t\geq 0,
\end{align*}
which we will use later.
\begin{remark}
Using the recurrence~\eqref{eqn:T_rec}, one can show by induction that
\[\deg T_n(x)=\lambda-\nu_p(m+1)\]
for $n\geq 1$,
where $\lambda\geq 0$ and $m\in \{0,\ldots,p^\lambda-1\}$ are chosen such that $n=cp^\lambda+m$ for some $c\in\{1,\ldots,p-1\}$.
\end{remark}
Let us compute some polynomials $T_n$ for $p=2$.
From the recurrence~\eqref{eqn:T_rec},
we obtain
\begin{align*}
  T_0(x) &= 1,&
  T_1(x) &= 2,\\
  T_2(x) &= 2 + x,&
  T_3(x) &= 4,\\
  T_4(x) &= 2 + x + 2x^2,&
  T_5(x) &= 4 + 2x,\\
  T_6(x) &= 4+ 2x + x^2,&
  T_7(x) &= 8,\\
  T_8(x) &= 2 + x + 2x^2 + 4x^3,&
  T_9(x) &= 4 + 2x + 4x^2.
\end{align*}
Note that $T_n(1)=n+1$, since the $n$-th row of Pascal's triangle contains $n+1$ entries.
Moreover, we define normalized generating functions $\overline T_n$:
\[\overline T_n(x)=\frac 1{\vartheta_p(0,n)}T_n(x).\]
By definition, we have $\left[x^0\right]\overline T_n(x)=1$.
We are extending these notations to finite words $v$ in $\{\tO,\ldots,\tN\}$ via the base-$p$ expansion: if $(v)_p=n$, we set $T_v \coloneqq T_n$ and $\overline T_v\coloneqq \overline T_n$.
Based on the polynomials $\overline T_n(x)$,
we shall define the rational functions $r_w$ occurring in the main theorem.
In order to do so, we define the \emph{left truncation} $w_L$ and the \emph{right truncation} $w_R$ on the set
$\widetilde W\cup\{\varepsilon\}$, as follows.
For $w\in \widetilde W$, $r\geq 0$, and digits $c\neq \tO$ and $a$, let
\begin{alignat*}{4}
\varepsilon_L&=\varepsilon,&\quad
(c\tO^r)_L&=\varepsilon,&\quad
(c\tO^rw)_L&=w;\\
\varepsilon_R&=\varepsilon,&\quad
c_R&=\varepsilon,&\quad
(wa)_R&=w.
\end{alignat*}
In other words, for $w\in \widetilde W$ the word $w_L$ is the longest proper suffix $u$ of $w$ 
such that $u\in \widetilde W\cup\{\varepsilon\}$.
Analogously, $w_R$ is the longest proper prefix $u$ of $w$ such that $u\in \widetilde W\cup\{\varepsilon\}$.
Note that we have $(w_L)_R=(w_R)_L$ for all $w\in \widetilde W\cup\{\varepsilon\}$;
we write $w_{LR}$ for the common value.
In what follows,
we write $\overline T_w \equiv \overline T_w(x)$ as a shorthand.
The following proposition,
a telescoping product,
is the first out of two pillars on which the main theorem rests.
\begin{proposition}\label{prp:telescope}
Let $v\in \widetilde W\cup \{\varepsilon\}$.
Then we have the identity
\begin{equation}\label{eqn:telescope}
\overline T_v=\prod_{w\in \widetilde W}\biggl(\frac{\overline T_w\overline T_{w_{LR} } }{\overline T_{w_R}\overline T_{w_L}}\biggr)^{\abs{v}_w}.
\end{equation}
\end{proposition}
We note that we do not use the explicit definition of $\overline T_w$ in the proof of this proposition.
We only need the property $\overline T_w(0)=1$, so that we may take quotients.
In particular, we will show that the product reduces to the fraction $\overline T_v/\overline T_\varepsilon$ by cancelling identical factors.
The following example clarifies this point.
\begin{example}
Let $p=2$ and $v=\tL\tO\tO\tL\tO$.
Then we have
\[\frac{\overline T_v}{\overline T_\varepsilon}=
\biggl(\frac{\overline T_{\tL}\overline T_\varepsilon}{\overline T_\varepsilon\overline T_\varepsilon}\biggr)^2
\biggl(\frac{\overline T_{\tL\tO}\overline T_\varepsilon}{\overline T_\tL\overline T_\varepsilon}\biggr)^2
\biggl(\frac{\overline T_{\tL\tO\tO}\overline T_\varepsilon}{\overline T_{\tL\tO}\overline T_\varepsilon} \biggr)
\biggl(\frac{\overline T_{\tL\tO\tO\tL}\overline T_\varepsilon}{\overline T_{\tL\tO\tO}\overline T_\tL} \biggr)
\biggl(\frac{\overline T_{\tL\tO\tO\tL\tO}\overline T_\tL}{\overline T_{\tL\tO\tO\tL}\overline T_{\tL\tO} } \biggr).
\]
\end{example}
For each word $w\in\widetilde W$ we can finally define the rational generating function
\begin{equation}\label{eqn:rw}
r_w(x)\coloneqq\frac{\overline T_w(x)\overline T_{w_{LR}}(x) }{\overline T_{w_R}(x)\overline T_{w_L}(x)}.
\end{equation}
We note that $r_w(x)=1$ for $w\in \widetilde W\setminus W$, which follows from the facts that $\overline T_a=1$ for $a\in\{\tL,\ldots,\tN\}$ and $\overline T_{v\tN}=\overline T_v$ for $v\in\widetilde W$, see~\eqref{eqn:T_rec}.
Now that we know $r_w$,
our main theorem can be stated completely explicitly.
\begin{theorem}\label{thm:main}
Let $w^{(1)},\ldots,w^{(\ell)}$ be admissible words and
$k_1,\ldots,k_{\ell}$ positive integers.
Assume that $c_j$ is the coefficient of the monomial
\[X_{w^{(1)}}^{k_1}\cdots X_{w^{(\ell)}}^{k_{\ell}}\]
in the polynomial $P_j$.
Then
\[
\sum_{j\geq 0}c_jx^j
=
\frac 1{k_1!}\bigl(\log r_{w^{(1)}}(x)\bigr)^{k_1}\cdots
\frac 1{k_{\ell}!}
\bigl(\log r_{w^{(\ell)}}(x) \bigr)^{k_{\ell}}.
\]
\end{theorem}
We list the first few rational functions $r_w$ for the case $p=2$ and admissible $w$:
\begin{align*}
  r_{\tL\tO}(x) &= 1+\tfrac 12 x,&
  r_{\tL\tO\tO}(x) &= 1+\frac{x^2}{1+x/2},\\
  r_{\tL\tL\tO}(x) &= 1+\frac{\frac 14 x^2}{1+x/2},&
  r_{\tL\tO\tO\tO}(x) &= 1+\frac{2x^3}{1+x/2+x^2},\\
  r_{\tL\tO\tL\tO}(x) &= 1+\frac{\frac 12x^3}{(1+x/2)^2},&
  r_{\tL\tL\tO\tO}(x) &= 1+\frac{\frac 12x^3}{(1+x/2+x^2)(1+x/2+x^2/4)}.
\end{align*}
As a straightforward application of Theorem~\ref{thm:main}
we obtain the corollary from the introduction, which we restate here.
\begin{corollary}\label{cor:log32}
Let $p=2$.
The coefficient of $X_{\tL\tO}$ in the polynomial $P_j$ equals
$\left[x^j\right]\log(1+x/2)$.
In particular,
\[\sum_{j\geq 0}c_j=\log(3/2).\]
\end{corollary}
\begin{proof}
  In this simple case all we need is $r_{\tL\tO}(x) = \overline T_2(x) = 1 + \frac{x}{2}$,
which does not have a singularity or a zero in the closed unit disc.
\end{proof}
\begin{remark}
The first step in finding our main theorem was to investigate the case $X_{\tL\tO}$.
From the first values $0,1/2,-1/8,1/24,-1/64,1/160$ it can be guessed easily that the corresponding generating function is $\log(1+x/2)$.
More generally, by considering integers $n(a)$ whose binary expansion is built of blocks as in the proof of Proposition~\ref{prp:unique_polynomial},
we obtained the conjecture that the generating function for $X_w$ is given by $\log \circ\,r_w$, where $r_w$ is given by~\eqref{eqn:rw}.
Finally we observed experimentally, using again the data obtained by Rowland's package,
that the generating function for $\mathfrak m_1\mathfrak m_2$ (where $\mathfrak m_1$ and $\mathfrak m_2$ are monomials) seems to be obtained by multiplying the generating functions for $\mathfrak m_1$ and $\mathfrak m_2$,
and some factor taking care of multiplicities.
This led us to the formulation of Theorem~\ref{thm:main}.
\end{remark}
We continued the computation of the rational functions $r_w$
and performed analogous experiments for the prime numbers $3,5,7$
in order to obtain a conjecture on the structure of $r_w$.
The statement of the following proposition is the result of these experiments
and constitutes the second main ingredient in the proof of our main theorem.
The proof can be found at the end of this paper.
\begin{proposition}\label{prp:r_numerator}
Let $p$ be a prime and assume that $w=w_{\mu-1}\cdots w_0\in W$.
The rational function $r_w(x)$ satisfies
\[r_w(x)=1+\frac{\alpha x^{\mu-1}}{\overline T_{w_L}(x)\overline T_{w_R}(x)},\]
where
\begin{equation}\label{eqn:alpha_def}
\alpha=
p^{\mu-2}
\frac{w_{\mu-1}}{w_{\mu-1}+1}
\cdot
\frac{p-w_0-1}{w_0+1}
\prod_{2\leq d\leq p}
d^{-2\lvert w'\rvert_{d-\tL}}
,
\end{equation}
and $w'=w_{\mu-2}\cdots w_1$.
\end{proposition}
\begin{remark}
Consider the special case $w=ca$ of this proposition.
We obtain $\alpha=\frac{c}{c+1}\frac{p-a-1}{a+1}$,
which gives the formula $\overline T_{ca}(x)=r_{ca}(x)=1+\frac{c}{c+1}\frac{p-a-1}{a+1}x$
(this also follows directly from the recurrence~\eqref{eqn:T_rec}).
By Theorem~\ref{thm:main} we obtain the coefficient of $X_{ca}$ in the polynomial $P_1$ by extracting the coefficient
\[
\left[x^1\right]\log\biggl(1+\frac{c}{c+1}\frac{p-a-1}{a+1}x\biggr)=\frac{c}{c+1}\frac{p-a-1}{a+1},
\]
which is consistent with~\eqref{eqn:carlitz_alt2}.
\end{remark}
The proof of Theorem~\ref{thm:main} is a combination of Propositions~\ref{prp:telescope} and~\ref{prp:r_numerator}, and consists of a series of identities.
\begin{proof}[Proof of Theorem~\ref{thm:main}]
By Proposition~\eqref{prp:telescope},
and the definition $[x^j] \overline T_n(x) = \vartheta_p(j,n)/\vartheta_p(0,n)$, we have
\[
\left[x^j\right]\prod_{w\in \widetilde W}r_w(x)^{\abs{n}_w}
=\frac{\vartheta_p(j,n)}{\vartheta_p(0,n)}
\quad
=P_j\Bigl(\bigl(\abs{n}_w\bigr)_{w\in W_j} \Bigr)
\]
for all $n\in\dN$.

Since $r_w(x)=1$ for $w=v\tN$ and $r_a(x)=1$ for $a\in\{\tL,\ldots,\tN\}$, words $w\in \widetilde W\setminus W$ do not contribute to the left hand side.
Moreover, 
Proposition~\ref{prp:r_numerator} implies that words $w\in W\setminus W_j$
do not contribute, since $\lvert w\rvert \geq j+2$
for these words and therefore $r_w(x)=1+\LandauO\bigl(x^{j+1}\bigr)$.
Let us reveal how the polynomial structure emerges in the left hand side.
The idea is to apply an exp-log decomposition on~\eqref{eqn:telescope}.
This is legitimate, as the constant term of $\overline T_n(x)$
and therefore of $r_w(x)$ is $1$, compare~\eqref{eqn:T_def}.
We have the identities
\begin{align*}
  \left[x^j\right]
  \prod_{w\in \widetilde W}r_w(x)^{\abs{n}_w}
  &=\left[x^j\right]
  \prod_{w\in W_j}r_w(x)^{\abs{n}_w}\\
  &=\left[x^j\right]\prod_{w \in W_j} \exp\bigl(  \abs{n}_w \log  r_w(x) \bigr)\\
  &=\left[x^j\right]\prod_{w \in W_j} \sum_{k \geq 0}
     \abs{n}_w^k \frac{\bigl(\log  r_w(x) \bigr)^k}{k!}\\
  &= \sum_{\substack{k_w\geq 0\\w\in W_j}}
     \Biggl(\bigl[x^j\bigr]\prod_{w\in W_j}\frac{\bigl(\log  r_w(x) \bigr)^{k_w}}{k_w!}\Biggr)\prod_{w\in W_j}\abs{n}_w^{k_w},
\end{align*}
where the last step is justified since there are only finitely many summands contributing to the $j$-th coefficient.
(This is the case by the condition $r_w(0)=1$,
which implies $\log r_w(x) = \LandauO(x)$ for $x \to 0$).

The right hand side is a polynomial in $\abs{n}_w$ for $w\in W$,
and by the uniqueness result (Proposition~\ref{prp:unique_polynomial})
the theorem is proved.
\end{proof}
Note that the argument given in the proof also gives a new proof of existence of the polynomials $P_j$.
\begin{remark}\label{rem:110}
By Proposition~\ref{prp:r_numerator} we can determine exactly for which $j$ a given monomial occurs first.
Since $\overline T_w(0)=1$ for all admissible words $w$,
we have $r_w(x)=1+\alpha x^k+\LandauO\bigl(x^{k+1}\bigr)$, where $\alpha$ is given by~\eqref{eqn:alpha_def} and $k=\abs{w}-1$,
therefore $\log r_w(x)=\alpha x^k+\LandauO(x^{k+1})$.
By Theorem~\ref{thm:main} the monomial $X_w$ occurs first in the polynomial $P_j$, where $j=\abs{w}-1$.
More generally, the monomial $X_{w^{(1)}}\cdots X_{w^{(\ell)}}$
(repetitions allowed) occurs first in $P_j$,
where $j=\abs{w^{(1)}}+\cdots+\abs{w^{(\ell)}}-\ell$.
That is, the lower bound for the first occurrence of a monomial given by~\eqref{eqn:RBG_implicit} is sharp.

We note that this observation is not sufficient to determine the number of terms in $P_j$;
in the generating function appearing in Theorem~\ref{thm:main} some higher coefficients may vanish.
This is for example the case for $w=\tL\tL\tO$. We have
\[\log r_{\tL\tL\tO}(x)
=
\log\biggl(\frac{1-(x/2)^3}{1-(x/2)^2}\biggr)=
\sum_{i\geq 1}\frac{x^{2i}}{i4^i}
-\sum_{i\geq 1}\frac{x^{3i}}{i8^i},
\]
and consequently the monomial $X_{\tL\tL\tO}$ does not occur in $P_j$ for $j=6\ell\pm 1$, where $\ell\geq 1$.
It is however true that each nontrivial monomial occurs in infinitely many $P_j$.
\end{remark}
\begin{corollary}\label{cor:inf}
Each monomial $X_{w^{(1)}}^{k_1}\cdots X_{w^{(\ell)}}^{k_{\ell}}$
except for the constant term~$1$ occurs in infinitely many $P_j$.
\end{corollary}
\begin{proof}
By Theorem~\ref{thm:main} the claim is equivalent to the statement that the power series $\prod_{i=1}^{\ell}\bigl(\log r_{w^{(i)}}(x)\bigr)^{k_i}$ is not a polynomial.
We will analyse the possible singularities,
which will contradict a polynomial behaviour.

Assume that $\rho_i$ is the radius of convergence of the power series
$\log r_{w^{(i)}}(x)$
and choose $j\in\{1,\ldots,\ell\}$ such that $\rho_j=\min_{1\leq i\leq \ell}\rho_i$,
moreover let $x_j$ be a singularity of $\log r_{w^{(j)}}(x)$ on the circle $\{x:\abs{x}=\rho_j\}$.
By Proposition~\ref{prp:r_numerator} we have $0<\rho_j<\infty$, and that the power series $\log r_{w^{(i)}}(x)$ does not have a zero apart from $x=0$.
Therefore the singularities cannot cancel,
which implies that $x_j$ is a singularity of
$\bigl(\log r_{w^{(1)}}(x)\bigr)^{k_1}\cdots \bigl(\log r_{w^{(\ell)}}(x)\bigr)^{k_{\ell}}$.
Consequently, this expression is not a polynomial.
\end{proof}
Moreover, we want to derive an asymptotic estimate on the number of terms in $P_j$, using Proposition~\ref{prp:r_numerator}.
\begin{corollary}\label{cor:number_of_terms}
The number of terms $N_j$ in the polynomial $P_j$ satisfies the bound
\[N_j\leq \left[x^j\right]\frac{1}{1-x}\exp\Biggl(\sum_{k\geq 1}\frac 1k\frac{(p-1)^2x^k}{1-px^k}\Biggr).\]
Asymptotically, for $j\rightarrow\infty$, 
this upper bound is
\begin{align*}
   \frac{e^{\mu (\sigma - 1/2)} }{2 p \mu^{1/4} \sqrt{\pi} } \frac{e^{2\sqrt{\mu j}} p^j}{j^{3/4}} \left(1 + \LandauO \left(\frac{1}{\sqrt{n}} \right)\right),
\end{align*}
with the constants $\mu = \frac{(p-1)^2}{p}$ and $\sigma = \sum_{k \geq 2} \frac{1}{k} \frac{1}{p^{k-1}-1}$.
Moreover, 
we have
\[N_j=\Theta\bigl(p^j e^{2\sqrt{\mu j}} j^{-3/4}\bigr).\]
The same estimates are true for the number $N'_j$ of terms in the polynomials $P'_j$.
\end{corollary}
\begin{proof}
The terms in $P_j$ are built from the variables in $W_j$, see~\eqref{eqn:Wj_def}.
In $W=\bigcup_{j\geq 1}W_j$ there are $p^{k-1}(p-1)^2$ many words $w$ of weight equal to $k$, for $k\geq 2$.
Here the \emph{weight} of a word $w$ is defined by $\abs{w}-1$.
The corresponding generating function is $\mathcal W(x)=(p-1)^2\frac{x}{1-px}$.

First, we want to determine 
the number of monomials having total weight $j$.
These are the monomials that, by~\eqref{eqn:RBG_implicit},
may appear in $P_j$, but cannot appear in $P_{j-1}$.
We obtain therefore the maximal number of ``new'' monomials in $P_j$.

A monomial is nothing else but a multiset of variables in $W$.
Thus, by the multiset construction (see~\cite[page 27]{FS2009})
we obtain the $exp$-part of the generating function in the corollary.
Finally, the factor $\frac{1}{1-x}$ stems from the fact that
also monomials from $P_0,\ldots,P_{j-1}$ are allowed in $P_j$.

For the asymptotic result, we first need to find the dominant singularity,
i.e., the one closest to the origin.
Note that the possible singularities are at $\omega_k^\ell p^{-1/k}$,
for $\ell = 0,\ldots,k-1$,
where $\omega_k = \exp(2 \pi i/k)$ is a $k$-th root of unity.
As $p \geq 2$, the dominant one is found at $1/p$ for $k=1$.
Thus, we may decompose our generating function into
\begin{align*}  \exp\left( \frac{(p-1)^2 x }{ 1- xp} \right) S(x), \end{align*}
where $S(x)$ is the generating function of the remaining factors.
The crucial observation is that $S(x)$ is analytic for $|x| < 1/\sqrt{p}$,
hence, for $|x| < 1/p$.
This is a well-known type of function for which a complete asymptotic expansion is known.
Using Wright's result from~\cite[Theorem~2]{W1933} we get the final result.
The constants are coming from $S(1/p)$.
The last statement follows from Proposition~\ref{prp:r_numerator} and the asymptotic statement,
since all monomials of weight~$j$ actually appear in $P_j$ with a nonzero coefficient,
and their number is a positive portion of the asymptotic main term.
\end{proof}
This type of function was already intensively considered in the literature.
It appears in the enumeration of permutations.
The analysis builds on a saddle point method,
see \cite[Example~VIII.7, p.~562]{FS2009}.
Wright~\cite{W1933} derived the asymptotics for the general form of an exponential singularity we encounter here,
extending the work of Perron~\cite{P1914}.
\begin{remark}\label{rem:number_of_terms}
We note that for the upper bound in Corollary~\ref{cor:number_of_terms} we do not need Proposition~\ref{prp:r_numerator},
but it suffices to use Rowland's paper, see~\eqref{eqn:RBG_implicit}.
The lower bound however uses Proposition~\ref{prp:r_numerator},
which implies that all monomials of weight $j$ do occur in the polynomial $P_j$.

For the prime $p=2$,
we implemented the method of finding the coefficients of $P_j$ by Theorem~\ref{thm:main}
in the Sage Mathematics Software System~\cite{S}.
In particular, we retrieve the formulas for $\vartheta_2(2,n),\ldots,\vartheta_2(4,n)$ obtained by Howard~\cite{H1971}, Spearman and Williams~\cite{SW1999} and Rowland~\cite{R2011} before.
Computing $P_0,\ldots,P_{11}$ took less than five minutes using our implementation,
which is a significant improvement over Rowland's algorithm~\cite{R2011}.

In the following table we compare the actual number of nonzero coefficients in $P_j$ (first line of numbers) with the upper bound from Corollary~\ref{cor:number_of_terms} (second line).
The number of nonzero coefficients in $P_j$ is sequence~\texttt{A275012} in Sloane's OEIS. Rowland notes (see~\texttt{A001316, A163000, A163577} in the OEIS, which are the sequences $n\mapsto\vartheta_2(j,n)$ for $j=0,1,2$) that these numbers give a measure of complexity of the sequences $n\mapsto \vartheta_2(j,n)$.
\[
\begin{array}{cccccccccccc}
P_0&P_1&P_2&P_3&P_4&P_5&P_6&P_7&P_8&P_9&P_{10}&P_{11}\\
1&1&4&11&29&69&174&413&995&2364&5581&13082\\
1&2&5&12&30&72&176&420&1005&2378&5611&13144
\end{array}
\]
From this numerical evidence it seems reasonable to conjecture that the upper bound given in Corollary~\ref{cor:number_of_terms} gives in fact the asymptotic main term of the number~$N_j$ of nonzero coefficients of $P_j$.
However, the exact behaviour of the integers~$N_j$ seems to be difficult to grasp, and remains an open problem at the moment.
\end{remark}
\subsection{Asymptotic behaviour of coefficients of a given monomial}\label{sec:asymp}
In this chapter we study the different asymptotic behaviours exhibited by a sequence $(c_j)_{j\geq 0}$ of coefficients of a monomial.
More precisely, we restrict ourselves to $p=2$ and monomials $X_w$ for $w\in W$.
The following lemma explains how the coefficients of the logarithm of a rational function behave asymptotically.
We will apply it repeatedly in the subsequent discussion.
\begin{lemma}[Coefficient asymptotics of $\log \circ \operatorname{rat}$]
  \label{lem:coefflograt}
  Let $r(x)$ be a rational function defined at $0$ such that $r(0)=1$.
  Choose $L\geq 0$, $\varepsilon_0,\ldots,\varepsilon_{L-1}\in \dZ\setminus\{0\}$ and pairwise different $\xi_0,\ldots,\xi_{L-1}\in\dC\setminus\{0\}$ in such a way that
 \[r(x)=(1-\xi_0 x)^{\varepsilon_0}\cdots (1-\xi_{L-1} x)^{\varepsilon_{L-1}}.\]
  (Note that this decomposition is unique up to the order of the factors.)
  Then
  \begin{equation}\label{eqn:coeffexactlograt}
    [x^n] \log r(x)=-\frac 1n\sum_{0\leq i<L} \varepsilon_i \xi_i^n
  \end{equation}
  for $n\geq 1$.
  In particular, assume without loss of generality that
  $\xi_0,\ldots,\xi_{m-1}$, for some $1\leq m\leq L$,
  have maximal absolute value among the $\xi_i$, and $M=\abs{\xi_0}$.
  Then
  \begin{equation*}
    [x^n] \log r(x)=-\frac 1n\sum_{0\leq i<m} \varepsilon_i \xi_i^n+\LandauO\bigl((M-\varepsilon)^n\bigr)
  \end{equation*}
  for some $\varepsilon>0$.
  If moreover $m=1$, we have for all $k\geq 1$
  \begin{equation}\label{eqn:coeffasylogratprod}
    [x^n] \bigl(\log r(x)\bigr)^{k} =
    k(-\varepsilon_0)^k \bigl(\log n\bigr)^{k-1}\frac{\xi_0^n}{n} \left(1 + \LandauO\left(\frac{1}{n}  \right) \right).
  \end{equation}
\end{lemma}
\begin{proof}
The first two statements follow immediately from the identity
\begin{equation*}
  [x^n] \log\left(\frac 1{1-x}\right) = [x^n] \sum_{n \geq 1} \frac{x^n}n = \frac 1n.
\end{equation*}
The asymptotic statements can be proved using standard results from singularity analysis (see Flajolet and Sedgewick~\cite{FS2009}).
We begin with the case $m=1$.
First of all, the location of the dominant singularity (the one closest to the origin) is responsible for the exponential growth of the coefficients.
Next note that the function $\log r(x)$ is singular if the rational function is either singular, or takes the value $0$.
If we assume that $\varepsilon_0>0$, the dominant singularity comes from the zero $1/\xi_0$ of the numerator of $r(x)$,
and the exponential growth of the $n$-th coefficient is given by $\xi_0^n$.
More precisely, a Taylor expansion of $r(x)$ at $x=r$ shows that
\begin{align*}
  \log\left( r(x) \right)
    &= \log\left( h(x)(x-r)^{d_r} \right)
     = -d_r \log\left( \frac{1}{1-x/r} \right) + \log(h(x)),
\end{align*}
where $\log(h(x))$ is analytic for $|x| \leq |r|+\varepsilon$.
If $\varepsilon_0<0$, we simply swap numerator and denominator of $r(x)$ and adjust the sign.
If $m > 1$ one deals separately with the different singularities.

If higher powers of the logarithm are considered we have to deal with Cauchy products.
In this case one can elementarily show the appearance of the  $\bigl(\log n\bigr)^{k-1}$ terms by partial summation combined with $  \sum_{k=1}^n \frac{1}{k}
      = \log n + \LandauO\left(1\right).  $
For more details we refer to \cite[Chapter VI]{FS2009}.
\end{proof}
\begin{examples}
  Let $p=2$ and consider $\log (r_{\tL\tL\tO}(x)) = \log\Bigl(\frac{1+x/2+x^2/4}{1+x/2}\Bigr)$.
Here, the numerator has the two roots $2 e^{2\pi i/3}$ and $2 e^{-2 \pi i /3}$,
whereas the denominator has the root $-2$.
In this case all roots lie on the same circle $|x|=2$, and therefore cancellations take place (compare Remark~\ref{rem:110}).
By~\eqref{eqn:coeffexactlograt} we obtain
\begin{align*}
  [x^n]\log r_{\tL\tL\tO}(x)
    &= \frac{2^{-n}}{n} \left((-1)^n - e^{2\pi i n/3} - e^{- 2\pi i n/3}\right).
\end{align*}
In this special case we have equality, as no other roots are involved.
Since the radius of convergence is larger than $1$, we can obtain the infinite sum of coefficients $c_j$ of $X_{\tL\tL\tO}$ by inserting $1$ into the generating function:
\begin{align*}
\sum_{j\geq 0}c_{j}
=\sum_{j\geq 0}\left[x^j\right]\log r_{\tL\tL\tO}(x)
&=\lim_{j\rightarrow\infty}\left[x^j\right]\frac{\log r_{\tL\tL\tO}(x)}{1-x}\\
&=\log r_{\tL\tL\tO}(1)
=\log(7/6).
\end{align*}

Now we consider the generating function $\frac 12\bigl(\log(1+x/2)\bigr)^2$ corresponding to the coefficients $c_j$ of $X_{\tL\tO}^2$.
In this case we have, by~\eqref{eqn:coeffasylogratprod},
\[    c_{j}=\frac {(-1)^j\log j}{j\cdot 2^j}\bigl(1+\LandauO(1/j)\bigr).    \]
In this simple case an exact form of the coefficients can be obtained from~\eqref{eqn:coeffexactlograt},
using the Cauchy product of
\[    \log r_{\tL\tO}(x)=\sum_{j\geq 1}\frac {(-1)^j}{j\cdot 2^j}x^j    \]
with itself:
\[c_j=\left[x^j\right]\frac 12\bigl(\log r_{\tL\tO}(x)\bigr)^2
=\frac {(-1)^j}{2^{j+1}}
\sum_{\substack{i_1,i_2\geq 1\\i_1+i_2=j}}\frac 1{i_1i_2}.    \]
Moreover, similarly as in the first example we have
\[    \sum_{j\geq 0}c_{j}=\frac 12\bigl(\log(3/2)\bigr)^2.    \]
\end{examples}
Let us now consider special classes of monomials, whose generating functions have a large radius of convergence and can be evaluated at $x=1$.
\begin{corollary}\label{cor:1s0}
  Consider the words $w = \tL^s \tO$ or $w = \tL^{4s+1} \tO\tO$ for $s \geq 1$. For fixed word $w$ and an integer $k\geq 0$ let $c_j$ be the coefficient of the corresponding monomial $X^k_w$.
  Then the radius of convergence of $\sum_{j\geq 0} c_j x^j$
  is greater than $1$
  (more precisely, equal to $2$ for the first family of values).
  Thus,
  \begin{align*}
    \sum_{j\geq 0} c_j &= \frac{1}{k!}\bigl(\log r_w(1)\bigr)^k.
  \end{align*}
\end{corollary}
\begin{proof}
  By the main theorem the considered generating function is given by $\frac{1}{k!}\log\bigl(r_w(x)\bigr)^k$. Let us start with the first family of words.
  We need to analyse the rational function $r_w(x) = \frac{T_{\tL^s\tO}(x)}{T_{\tL^{s-1}\tO}(x)}$, as our plan is to apply Lemma~\ref{lem:coefflograt}.
It is not difficult to show that
  \begin{align*}
    T_{\tL^s\tO}(x) &= \frac{1 - \left(x/2\right)^{s+1}}{1-x/2}.
  \end{align*}
  Thus,
  $
    r_w(x) = \frac{1 - \left(x/2\right)^{s+1}}{1 - \left(x/2\right)^{s}},
  $
  and we see that all roots of the numerator and the denominator are located on the circle $|x| = 2$.

  For the second family of words, we get
  \begin{align*}
    T_{\tL^r \tO\tO}(x) &= \frac{q_{r+1}(x/2)}{q_r(x/2)} \cdot \frac{1-(x/2)^{r}}{1-(x/2)^{r+1}}, & \text{ with } &&
    q_r(t) &= 4t^{r+1} + t^r - 4t^2-1.
  \end{align*}
  Hence, we are interested in the roots of the polynomials $q_r(x)$. By Rouche's Theorem there are exactly $2$ roots inside the disc
 $|t|<2^{-1} (1 + 2^{-r+2})$. These two are very close to $\pm i/2$.
 In particular, by Newton's method starting with $i/2$,
 we get after one iteration the very good approximation
  \begin{align*}
    \frac{i}{2} + \left(\frac{i}{2}\right)^{r} \left(\frac{1}{2} - \frac{i}{4}\right) + O\left(\frac{1}{2^{2r}}\right).
  \end{align*}
  Therefore, the roots of $q_r(t)$ are in absolute value greater than $1/2$ for $r \equiv 1,2 \mod 4$ and less than $1/2$ for $r \equiv 0,3 \mod 4$. In particular, for $r \equiv 1 \mod 4$ we have that the roots of $q_{r+1}(x/2)$ and $q_{r}(x/2)$ are both in absolute value greater than $1$. Thus, the radius of convergence is larger than $1$, and it is legitimate to insert $1$.
\end{proof}
By Lemma~\ref{lem:coefflograt} the sequence of coefficients $(c_j)_{j \geq 0}$
for a given word $w$ can exhibit different kinds of behaviours,
corresponding to the position of the zeros and singularities of $r_w(x)$.
Because of the construction of $r_w(x)$, there is a convergence--divergence dichotomy, which we summarize in the following corollary.
\begin{corollary}\label{cor:behaviours}
  Let $w\in W$ and write
$r_w(x)=(1-\xi_0 x)^{\varepsilon_0}\cdots (1-\xi_{L-1} x)^{\varepsilon_{L-1}}$
  with pairwise different, nonzero $\xi_i\in\dC$ and nonzero $\varepsilon_i\in\dZ$,
  such that $\abs{\xi_0}\geq \cdots\geq\abs{\xi_{L-1}}$.
\begin{enumerate}[(a)]
\item \label{item:sing_a}
If $\abs{\xi_0}\leq 1$, the sequence $c_w$ converges, moreover
we have the convergent series
\[\sum_{j\geq 0}c_j=\log r_w(1).\]
\item \label{item:sing_b}
If $\abs{\xi_0}>1$, the sequence $c_w$ diverges.
If moreover $1/\xi_0$ is the only dominant singularity, then
$\xi_0$ is a real number in 
$(-\infty,-1]$, 
and we have $c_w(j)\sim -\varepsilon_0\xi_0^j/j$.
\end{enumerate}
\end{corollary}
\begin{proof}
The case $\abs{\xi_0}<1$ is clear, since the function $\log r_w(x)$ has no singularity in the closed unit disc in this case.
For the case $\abs{\xi_0}=1$ we note that $\xi_i\neq 1$ for all $i$, since $T_v$ has only positive coefficients.
Since the sum $\sum_{j\geq 1}\xi^j/j$ converges for all $j$ on the unit circle such that $j\neq 1$, the sum
$\sum_{j\geq 1} c_j$ converges by~\eqref{eqn:coeffexactlograt}.
Abel's limit theorem finishes the proof for this case.
Finally, case~\eqref{item:sing_b} follows from Lemma~\ref{lem:coefflograt} and the positivity of coefficients of $T_v$.
\end{proof}
In the following, let $p=2$.
We have seen (Corollaries~\ref{cor:log32} and~\ref{cor:1s0}) that
case~(\ref{item:sing_a}) occurs for $w=\tL^s\tO$, where $s\geq 1$.

Case~(\ref{item:sing_b}) appears for $w=\tL\tO\tL\tO$
(dominant singularity at $x_0\sim -0.86408$).
In this case the singularity is coming from the logarithm,
as $r_w(x_0) = 0$.
This is also called a supercritical composition scheme,
as the outer function is responsible for the singularity, see~\cite[Chapter~VI.9]{FS2009}.

This case also appears for $w=\tL\tO\tL\tO\tO$
(dominant singularity again at $x_0\sim -0.86408$).
In this case however,
the denominator of $r_w$ is zero at $x_0$,
thus the singularity is coming from a simple pole.
This is also called a subcritical composition scheme,
as the inner function is responsible for the singularity.

By approximate computation of the roots of $\overline T_v$
using GNU~Octave~\cite{O} we determined all words of length at most $10$ for which case~\eqref{item:sing_a} occurs.
Besides for the words of the form $\tL^s\tO$ or $\tL^{4s+1}\tO\tO$,
this also seems to be the case for the words
$\tL^s\tO\tL^t\tO$, where $s\geq 1$ and $t\geq 2$.
Here is the list of remaining words $w\in W$ of length at most $10$,
not falling into one of these three classes, for which this case occurs as well:
\[
\begin{array}{rrrr}
\tL\tO\tO\tL\tL\tL\tL\tO,&
\tL\tO\tL\tL\tO\tL\tL\tL\tO,&
\tL\tO\tL\tL\tL\tO\tL\tL\tO,&
\tL\tO\tL\tL\tL\tL\tO\tL\tO,\\
\tL\tO\tL\tL\tL\tL\tL\tO\tO,&
\tL\tL\tL\tO\tL\tL\tO\tL\tO,&
\tL\tO\tL\tL\tO\tL\tL\tL\tL\tO,&
\tL\tO\tL\tL\tL\tO\tL\tL\tL\tO,\\
\tL\tO\tL\tL\tL\tL\tO\tL\tL\tO,&
\tL\tL\tO\tL\tL\tO\tL\tL\tL\tO,&
\tL\tL\tO\tL\tL\tL\tO\tL\tL\tO,&
\tL\tL\tO\tL\tL\tL\tL\tO\tL\tO,\\
\tL\tL\tO\tL\tL\tL\tL\tL\tO\tO,&
\tL\tL\tL\tL\tO\tL\tL\tO\tL\tO.
\end{array}
\]
We leave the classification of the words $w\in W$ for which the sum $\sum_{j\geq 0}c_j$ converges as an open problem.
\subsection{A simplified recurrence for $\vartheta_p(j,n)$}\label{sec:simplified}
Rarefying $\vartheta_p(j,n)$ in the first coordinate by the factor $p-1$,
and shifting $j$ by $s_p(n)$,
the recurrence~\eqref{eqn:theta_rec} is transformed into a simpler form:
the term $\nu_p$ disappears,
instead the maximal shift occurring in the first coordinate is $2p-2$.
We pass to the details.
Define, for $k,n\geq 0$,
\begin{align*}
\tilde\vartheta_p(k,n) &=
\begin{cases}
  \vartheta_p\Bigl(\frac{k-s_p(n)}{p-1},n\Bigr), & k\geq s_p(n)\text{ and }p-1\mid k-s_p(n);\\
  0,                     & \text{otherwise}.
\end{cases}
\end{align*}

Setting for simplicity $\tilde\vartheta_p(k,n)=0$ if $k<0$ or $n<0$,
we obtain the following recurrence relation for $k,n\geq 0$,
where we use the Kronecker delta,
which is defined by $\delta_{i,i} = 1$, and $\delta_{i,j}=0$ for $i \neq j$.
\begin{equation*}
\begin{aligned}
\tilde\vartheta_p(0,n)&=\delta_{0,n},&n\geq 0;\\
\tilde\vartheta_p(k,0)&=\delta_{k,0},&k\geq 0,
\end{aligned}
\end{equation*}
and for $n\geq 0$ and $0\leq a<p$,
\begin{equation*}
\tilde\vartheta_p(k,pn+a)=(a+1)\tilde\vartheta_p(k-a,n)+(p-a-1)\tilde\vartheta_p(k-p-a,n-1).
\end{equation*}
The proof of this new recurrence is straightforward and uses the identity
\begin{equation}\label{eqn:legendre_II}
s_p(n+1)-s_p(n)=1-(p-1)\nu_p(n+1),
\end{equation}
which follows immediately by writing $n$ in base $p$
and counting the number of times the digit $\tN$ occurs at the lowest digits of $n$,
and also the recurrence
\begin{align*}\label{eqn:sp_rec}
s_p(pn+a)=s_p(n)+a&&(0\leq a<p).
\end{align*}

In the Tables~\ref{tbl:tilde_theta_values2}--\ref{tbl:tilde_theta_values5} we list some coefficients of $\tilde\vartheta_p(k,n)$ for $p=2,3,5$, respectively.
\def\extrarowheight{1ex}             
\newcommand{\SetWidth}[1]{\makebox[1.6em]{$#1$}}%
\begin{table}[h!]
\[
\begin{array}{c|c@{}c@{}c@{}c@{}c@{}c@{}c@{}c@{}c@{}c@{}c@{}c@{}c@{}c@{}c@{}c@{}c@{}c}
&\SetWidth{0}&\SetWidth{1}&\SetWidth{2}&\SetWidth{3}&\SetWidth{4}&\SetWidth{5}&\SetWidth{6}&\SetWidth{7}&\SetWidth{8}&\SetWidth{9}&\SetWidth{10}&\SetWidth{11}&\SetWidth{12}&\SetWidth{13}&\SetWidth{14}&\SetWidth{15}&\SetWidth{16}&\SetWidth{17}\\
\hline
 0& 1&  &  &  &  &  &  &  &  &  &  &  &  &  &  &  &  &  \\
 1&  & 2& 2&  & 2&  &  &  & 2&  &  &  &  &  &  &  & 2&  \\
 2&  &  & 1& 4& 1& 4& 4&  & 1& 4& 4&  & 4&  &  &  & 1& 4\\
 3&  &  &  &  & 2& 2& 2& 8& 2& 2& 4& 8& 2& 8& 8&  & 2& 2\\
 4&  &  &  &  &  &  & 1&  & 4& 4& 1& 4& 5& 4& 4&16& 4& 4\\
 5&  &  &  &  &  &  &  &  &  &  & 2&  & 2& 2& 2&  & 8& 8\\
 6&  &  &  &  &  &  &  &  &  &  &  &  &  &  & 1&  &  &
\end{array}
\]
\caption{Some coefficients of $\tilde\vartheta_2(k,n)$.
The variable $k$ corresponds to the row number in this table.}
\label{tbl:tilde_theta_values2}
\end{table}
\def\extrarowheight{0ex}             
\def\extrarowheight{1ex}             
\renewcommand{\SetWidth}[1]{\makebox[1.6em]{$#1$}}%
\begin{table}[h!]
\[
\begin{array}{c|c@{}c@{}c@{}c@{}c@{}c@{}c@{}c@{}c@{}c@{}c@{}c@{}c@{}c@{}c@{}c@{}c@{}c}
&\SetWidth{0}&\SetWidth{1}&\SetWidth{2}&\SetWidth{3}&\SetWidth{4}&\SetWidth{5}&\SetWidth{6}&\SetWidth{7}&\SetWidth{8}&\SetWidth{9}&\SetWidth{10}&\SetWidth{11}&\SetWidth{12}&\SetWidth{13}&\SetWidth{14}&\SetWidth{15}&\SetWidth{16}&\SetWidth{17}\\
\hline
 0& 1&  &  &  &  &  &  &  &  &  &  &  &  &  &  &  &  &  \\
 1&  & 2&  & 2&  &  &  &  &  & 2&  &  &  &  &  &  &  &  \\
 2&  &  & 3&  & 4&  & 3&  &  &  & 4&  & 4&  &  &  &  &  \\
 3&  &  &  & 2&  & 6&  & 6&  & 2&  & 6&  & 8&  & 6&  &  \\
 4&  &  &  &  & 1&  & 4&  & 9&  & 4&  & 5&  &12&  &12&  \\
 5&  &  &  &  &  &  &  & 2&  & 6&  & 6&  & 4&  & 8&  &18\\
 6&  &  &  &  &  &  &  &  &  &  & 3&  & 4&  & 3&  & 4&  \\
 7&  &  &  &  &  &  &  &  &  &  &  &  &  & 2&  & 2&  &  \\
 8&  &  &  &  &  &  &  &  &  &  &  &  &  &  &  &  & 1&  \\
\end{array}
\]
\caption{Some coefficients of $\tilde\vartheta_3(k,n)$.
}
\label{tbl:tilde_theta_values3}
\end{table}
\def\extrarowheight{0ex}             
\def\extrarowheight{1ex}             
\renewcommand{\SetWidth}[1]{\makebox[1.6em]{$#1$}}
\begin{table}[h!]
\[
\begin{array}{c|c@{}c@{}c@{}c@{}c@{}c@{}c@{}c@{}c@{}c@{}c@{}c@{}c@{}c@{}c@{}c@{}c@{}c}
&\SetWidth{0}&\SetWidth{1}&\SetWidth{2}&\SetWidth{3}&\SetWidth{4}&\SetWidth{5}&\SetWidth{6}&\SetWidth{7}&\SetWidth{8}&\SetWidth{9}&\SetWidth{10}&\SetWidth{11}&\SetWidth{12}&\SetWidth{13}&\SetWidth{14}&\SetWidth{15}&\SetWidth{16}&\SetWidth{17}\\
\hline
 0& 1&  &  &  &  &  &  &  &  &  &  &  &  &  &  &  &  &  \\
 1&  & 2&  &  &  & 2&  &  &  &  &  &  &  &  &  &  &  &  \\
 2&  &  & 3&  &  &  & 4&  &  &  & 3&  &  &  &  &  &  &  \\
 3&  &  &  & 4&  &  &  & 6&  &  &  & 6&  &  &  & 4&  &  \\
 4&  &  &  &  & 5&  &  &  & 8&  &  &  & 9&  &  &  & 8&  \\
 5&  &  &  &  &  & 4&  &  &  &10&  &  &  &12&  &  &  &12\\
 6&  &  &  &  &  &  & 3&  &  &  & 8&  &  &  &15&  &  &  \\
 7&  &  &  &  &  &  &  & 2&  &  &  & 6&  &  &  &12&  &  \\
 8&  &  &  &  &  &  &  &  & 1&  &  &  & 4&  &  &  & 9&  \\
 9&  &  &  &  &  &  &  &  &  &  &  &  &  & 2&  &  &  & 6\\
\end{array}
\]
\caption{Some coefficients of $\tilde\vartheta_5(k,n)$.
}
\label{tbl:tilde_theta_values5}
\end{table}
\def\extrarowheight{0ex}             

We want to derive a product representation for $\tilde\vartheta_p(j,n)$.
In order to do so, we note the well-known fact due to Legendre
stating that
\begin{equation}\label{eqn:legendre_I}
\nu_p(n!)=\frac{n-s_p(n)}{p-1},
\end{equation}
for prime $p$.
This can be proved easily by summing the identity~\eqref{eqn:legendre_II}.
Applying~\eqref{eqn:legendre_I} three times, we obtain
\begin{equation}\label{eqn:legendre_III}
\nu_p\binom{n}{t}=\frac{s_p(n-t)+s_p(t)-s_p(n)}{p-1}.
\end{equation}
We note that, by Kummer's theorem~\cite{K1852},
the left hand side of~\eqref{eqn:legendre_III} is the number of borrows occurring in the subtraction $n-t$.
Let us define the bivariate generating function
$\widetilde{T}(x,z) \coloneqq \sum_{k,n\geq 0}\tilde\vartheta_p(k,n) x^kz^n.$
We will prove that $\widetilde T $ can be written compactly as an infinite product.
By the definition of $\tilde\vartheta$, 
the binomial coefficient $\binom nt$ contributes to $k=s_p(n)+(p-1)\nu_p\binom nt$.
Thus, we obtain by~\eqref{eqn:legendre_III}
  \begin{align*}
    \widetilde{T}(x,z)
      &= \sum_{n \geq 0} z^n \sum_{t=0}^n x^{s_p(n)+(p-1)\nu_p\binom nt}
      = \sum_{n \geq 0} z^n \sum_{t=0}^n x^{s_p(t) + s_p(n-t)}\\
      &= \Biggl(\sum_{n \geq 0} z^n x^{s_p(n)}\Biggr)^2
      = \prod_{i\geq 0}\Bigl(1+xz^{p^i}+x^2z^{2p^i}+\cdots+x^{p-1}z^{(p-1)p^i}\Bigr)^2,
  \end{align*}
where the last equality holds due the uniqueness of the base-$p$ expansion of an integer $n$.
This product representation should be compared to~\cite[Equations~(3.3), (3.12)]{C1967}.
Since Carlitz does not use the transformation in the first coordinate,
his product takes a more complicated form.
For $p=2$ we have the special case
\begin{align*}\label{eqn:tilde_T_product2}
  \sum_{k,n\geq 0}\tilde\vartheta_2(k,n) x^kz^n
&= \prod_{i\geq 0}\Bigl(1+xz^{2^i}\Bigr)^2.
\end{align*}
We note that this product representation can be used for an alternative proof of Carlitz' recurrence~\eqref{eqn:carlitz_rec}.

We finish this section with a remark on divisibility in \emph{columns} of Pascal's triangle.
\subsection{Divisibility in columns of Pascal's triangle}\label{sec:columns}
In the recent paper~\cite{DKS2016} by Drmota, Kauers, and the first author,
we deal with a conjecture by Cusick (private communication, 2012, 2015)
stating that
\[c_t\coloneqq\dens\{m\geq 0:s_2(m+t)\geq s_2(m)\}>1/2,\]
for all $t\geq 0$.
Here $\dens A$ denotes the asymptotic density of a set $A\subseteq \dN$, which exists in this case.
By~\eqref{eqn:legendre_III} this corresponds to a problem on divisibility in columns of Pascal's triangle:
if we define
$\textstyle \rho_2(j,t)=\dens\bigl\{m\geq 0:\nu_2\binom{m+t}{m}=j\bigr\}$%
\footnote{In~\cite{DKS2016}, we use the notations
$\delta(j,t)=\dens\bigl\{m\geq 0:s_2(m+t)-s_2(m)=j\}$ for all $j\in\dZ$,
and $b_{2^j}=\dens\bigl\{m:2^j\nmid\binom{m+t}{m}\}$.
We have
$\rho_2(j,t)=\delta(s_2(t)-j,t)$ for all $j\geq 0$
and
$b_{2^j}(t)=\rho_2(0,t)+\cdots+\rho_2(j-1,t)$ for $j\geq 1$.
},
the conjecture states that
\begin{equation*}
\sum_{j\leq s_2(t)}\rho_2(j,t)>1/2.
\end{equation*}
We gave in~\cite[Theorem~1]{DKS2016} a partial answer,
solving the conjecture for almost all $t$ in the sense of asymptotic density.
More precisely, we proved that for all $\varepsilon>0$,
\[\bigl\lvert\{t\leq T:1/2<c_t<1/2+\varepsilon\}\bigr\rvert=T+\LandauO(T/\log T).\]
The full statement of Cusick's conjecture is however still an open problem.
We also want to note the recent work by Emme and Hubert~\cite{EH2016} (preprint), which continues earlier work by Emme and Prikhod'ko~\cite{EP2015} (preprint).
They proved that for almost all $X\in\{0,1\}^\dN$ with respect to the balanced Bernoulli measure the values
\[\dens\bigl\{n\in \dN: s_2(n+a_X(k))-s_2(n)\leq x\sqrt{k/2}\bigr\}\]
converge pointwise to the standard normal distribution as $k\rightarrow\infty$,
where $a_X(k)=\sum_{0\leq j<k}X_j2^j$.

Surprisingly, the ``column densities'' $\rho_2(j,t)$ can be expressed by the same polynomial $P_j$ as the ``row counts'' $\vartheta_2(j,n)$ (see~\cite[Sections~3.2 and~3.3]{DKS2016}).
We have $\rho_2(0,t)=2^{-\abs{t}_\tL}$ and, for example,
\begin{align*}
\rho_2(1,t)/\rho_2(0,t)&=\frac 12\abs{t}_{\tO\tL},\\[2mm]
\rho_2(2,t)/\rho_2(0,t)&=-\frac 18\abs{t}_{\tO\tL}+\frac 18\abs{t}_{\tO\tL}^2
+\abs{t}_{\tO\tL\tL}+\frac 14\abs{t}_{\tO\tO\tL}.
\end{align*}
In general, if we denote by $\overline w$ the Boolean complement of the word $w\in W$, these expressions are obtained by inserting the value $\abs{t}_{\overline w}$ for the variable $X_w$ in $P_j$ (compare to~\eqref{eqn:theta_rep}):
\begin{equation*}\label{eqn:rho_rep}
t\mapsto \bigl(\abs{t}_{\overline w}\bigr)_{w\in W_j}\mapsto P_j\Bigl(\bigl(\abs{t}_{\overline w}\bigr)_{w\in W_j}\Bigr)
=\frac{\rho_2(j,t)}{\rho_2(0,t)}.
\end{equation*}
\section{Proofs}\label{sec:proofs}
\begin{proof}[Proof of Proposition~\ref{prp:unique_polynomial}]
Assume that $P_j$ and $\widetilde P_j$ are two polynomials in the variables $X_w$ ($w\in W$), representing $\vartheta(j,n)/\vartheta(0,n)$,
and let $R$ be the maximal degree with which a variable $X_w$ occurs in $P_j$ or $\widetilde P_j$.
Moreover, let $\ell$ be such that $\ell+1$ is the maximal length of a word $w$ such that the variable $X_w$ occurs in one of the polynomials.
The strategy is to compute the coefficients of a polynomial starting from its values.
For a multivariate polynomial in $M$ variables,
where the degree of each variable is bounded by $R$,
this can be done by evaluating the polynomial at each tuple in $\{0,\ldots,R\}^M$,
and applying recursively the fact that a univariate polynomial $q$ is determined by $\deg q+1$ of its values.
We adapt this strategy, taking the dependence between the variables into account.

On the set $W_\ell$ we have a partial order $\preceq$ defined by $v\preceq w$
if and only if $v$ is a factor of $w$.
For convenience, we extend this order to a total order on $W_\ell$ and denote it by the same symbol $\preceq$.
Let $w_0,\ldots,w_{M-1}$ be the increasing enumeration of $W_\ell$
(where $M=\abs{W_\ell}$).
We will work with certain ``test integers'', defined as follows.
For a vector $a=(a_m)_{m<M}$ in $\{0,\ldots,R\}^M$ let $n(a)$ be the integer whose binary expansion is given by the concatenation $v_{M-1}\cdots v_0$, where
\[v_m=\bigl(w_m\tN^\ell\tO^\ell\bigr)^{a_m}\bigl(\tN^\ell\tO^\ell\bigr)^{R-a_m}.\]
The idea behind this is that $\tN^\ell\tO^\ell$ acts as a ``separator''
in the sense that admissible factors of $n(a)$ of length $\leq \ell+1$
are contained completely in one of the building blocks
$w_m\tN^\ell\tO^\ell$ or $\tN^\ell\tO^\ell$. 
(At this point the restrictions $w_{\mu-1}\neq\tO$, $w_0\neq\tN$
for a word $w_{\mu-1}\cdots w_0\in W$ come into play.)
By varying the values $a_m$ we can therefore vary the factor count $\lvert\cdot\rvert_{w_m}$ without changing $\lvert\cdot\rvert_{w_{m'}}$ for $m'>m$.
For simplicity, we rename the variables $X_{w_m}$ to $X_m$.
We prove the following statement by induction on $s$.
\begin{claim}
Assume that $s$ is an integer, $0\leq s\leq M$.
For all $a_0,\ldots,a_{M-1}$, $k_0,\ldots,k_{s-1}\in\{0,\ldots,R\}$ we have
\[
\left[X_0^{k_0}\cdots X_{s-1}^{k_{s-1}}\right]
\Bigl(P_j-\widetilde P_j\Bigr)\Bigl(X_0,\ldots,X_{s-1},\abs{n(a)}_{w_s},\ldots,\abs{n(a)}_{w_{M-1} } \Bigr)=0.
\]
\end{claim}
The case $s=0$ follows from the assumption that $P_j$ and $\widetilde P_j$ yield the same value for all assignments $X_w=\abs{n}_w$, where $n\geq 0$.
The case $s=M$ is the desired statement that $P_j=\widetilde P_j$,
by the fact that the degree of each variable in $P_j$ and $\widetilde P_j$ is bounded by $R$.
Assume therefore that the statement holds for some $s<M$ and let $a_0,\ldots,a_{M-1},k_0,\ldots,k_{s-1}\in\{0,\ldots,R\}$.
We define polynomials $Q(X_s)$ 
and $\widetilde Q(X_s)$ 
in one variable, of degree at most $R$, by
\[
Q(X_s)=
\left[X_0^{k_0}\cdots X_{s-1}^{k_{s-1}}\right]
P_j\Bigl(X_0,\ldots,X_s,\abs{n(a)}_{w_{s+1}},\ldots,\abs{n(a)}_{w_{M-1} } \Bigr),\]
analogously $\widetilde Q$.
By the definition of the total order $\preceq$ we have
\[ \bigl \lvert n\bigl(a^{(r)} \bigr)\bigr \rvert_{w_m}=\bigl \lvert n(a)\bigr \rvert_{w_m}\]
for $0\leq r\leq R$ and $m>s$, where
\[
a^{(r)}_\ell=\begin{cases}a_\ell,&\ell\neq s;\\r,&\ell=s.\end{cases}
\]
By applying the induction hypothesis for $a^{(0)},\ldots,a^{(R)}$,
we obtain the equality $Q(N)=\widetilde Q(N)$ for the $R+1$ values
$\abs{n\bigl(a^{(0)}\bigr)}_{w_s},\ldots,\abs{n\bigl(a^{(R)}\bigr)}_{w_s}$ of $N$,
therefore
\begin{multline*}
0=\left[X_s^{k_s}\right](Q-\widetilde Q)(X_s)\\=
\left[X_0^{k_0}\cdots X_m^{k_s}\right]
\Bigl(P_j-\widetilde P_j\Bigr)\Bigl(X_0,\ldots,X_s,\abs{n(a)}_{w_{s+1}},\ldots,\abs{n(a)}_{w_{M-1} }\Bigr).
\end{multline*}
This proves that $P_j=\widetilde P_j$.
\end{proof}
\begin{proof}[Proof of Proposition~\ref{prp:telescope}]
Let $v\in \widetilde W\cup\{\varepsilon\}$.
The proof is by induction on the length of $v$,
the case $v\in\{\varepsilon,\tL,\ldots,\tN\}$ being trivial.
Moreover, 
for the words $c\tO^sa$,
where $c\in\{\tL,\ldots,\tN\}$, $s\geq 1$ and $a\in\{\tO,\ldots,\tN-\tL\}$,
we obtain
\[
\prod_{w\in \widetilde W}
\biggl(\frac{\overline T_w\overline T_{w_{LR} } }{\overline T_{w_R}\overline T_{w_L} } \biggr)^{\abs{v}_w}
=
\frac{\overline T_{c\tO^sa}}{\overline T_{c\tO^s}}
\cdot
\frac{\overline T_{c\tO^s}}{\overline T_{c\tO^{s-1} } }
\cdots
\frac{\overline T_{c\tO}}{\overline T_{c}}
\frac{\overline T_c}{\overline T_\varepsilon}
=\overline T_{c\tO^sa}.
\]
Suppose that the statement holds for some $v'\in \widetilde W$.
It is sufficient to show that it is also true for $v=a\tO^sv'$,
where $a\in\{\tL,\ldots,\tN\}$ and $s\geq 0$.

Since words in $\widetilde W$ do not start with the letter $\tO$ (read from left to right),
a factor of $v$ that is an element of $\widetilde W$ is either a factor of $v'$ or a prefix of $v$.
This implies that the product corresponding to $v$
is obtained from the product corresponding to~$v'$, multiplied by
$\overline T_w\overline T_{w_{LR}}/(\overline T_{w_R}\overline T_{w_L})$
for each prefix $w$ of $v$ such that $w\in \widetilde W$. This product of prefixes equals
\[
\prod_{\substack{w \text{ prefix of }v\\w\in \widetilde W} }
\frac{\overline T_w\overline T_{w_{LR} } }{\overline T_{w_R}\overline T_{w_L} }
=
\prod_{\substack{w \text{ prefix of }v\\w\in \widetilde W} }
\frac{\overline T_w}{\overline T_{w_R}}
\prod_{\substack{w \text{ prefix of }v'\\w\in \widetilde W} }
\frac{\overline T_{w_R}}{\overline T_w}
=\frac{\overline T_v}{\overline T_{v'}}.
\]
This shows the desired form and together with the induction hypothesis it yields the claim.
\end{proof}
Finally, we prove Proposition~\ref{prp:r_numerator}.
\begin{proof}[Proof of Proposition~\ref{prp:r_numerator}]
Assume that $w=w_{\mu-1}\cdots w_0\in W$.
The statement we want to prove is equivalent to
\begin{equation}\label{eqn:r_numerator_equivalent}
\overline T_w\overline T_{w_{LR}}-\overline T_{w_L}\overline T_{w_R}=\alpha x^{\mu-1},
\end{equation}
where
\[    \alpha=p^{\mu-2}\frac{w_{\mu-1}}{w_{\mu-1}+1}\frac{p-w_0-1}{w_0+1}
\prod_{2\leq d\leq p}d^{-2\lvert w'\rvert_{d-\tL}},    \]
and $w'$ is obtained from $w$ by omitting the left- and rightmost digits.
We want to prove the statement by induction on the \emph{right depth} of $w\in W$.
This is the number of right truncations needed to map $w$ to a \emph{base case}, which are words $v$ such that $v_L=\varepsilon$. Note that these are exactly the words of the form $v=c\tO^t$, where $c\in\{\tL,\ldots,\tN\}$ and $t\geq 0$.

We proceed to evaluating $\overline T_w\overline T_{w_{LR}}-\overline T_{w_L}\overline T_{w_R}$ for the base cases,
thus confirming~\eqref{eqn:r_numerator_equivalent} for these cases.
If $w=c\tO^t$,
where $t\geq 1$ and
$c\in\{\tL,\ldots,\tN\}$,
it follows by induction, using~\eqref{eqn:T_rec}, that
\[
\begin{aligned}
\overline T_{c\tO^t}(x)&=
1+\frac{p-1}{p}\frac{c}{c+1}\bigl((px)^1+\cdots+(px)^t\bigr),\\
\overline T_{c\tO^{t-1}}(x)&=
1+\frac{p-1}{p}\frac{c}{c+1}\bigl((px)^1+\cdots+(px)^{t-1}\bigr),
\end{aligned}
\]
therefore
\[
\begin{aligned}
\overline T_w\overline T_{w_{LR}}-\overline T_{w_L}\overline T_{w_R}
&=
\overline T_w-\overline T_{w_R}=
\frac{c}{c+1}(p-1)p^{t-1}x^t.
\end{aligned}
\]
Equation~\eqref{eqn:r_numerator_equivalent} therefore holds for the base cases.
Assume that we have already established the statement for all $w\in W$
having right depth $\leq d-1$, where $d\geq 1$,
and assume that $\widetilde w\in W$ has right depth equal to $d$.
Then $\widetilde w$ is of (exactly) one of the following forms,
for some nontrivial word $w\in\{\tO,\ldots,\tN\}^*$.
\begin{align}
wb\tO,&&&b\in\{\tL,\ldots,\tN\};\label{item:first_case}\\
wb\tO^t,&&&b\in\{\tL,\ldots,\tN\},t\geq 2;\label{item:second_case}\\
wa,&&&a\in\{\tL,\ldots,\tN-\tL\}.\label{item:third_case}
\end{align}

We will use the following auxiliary formulas.
If $wb\in \widetilde W$, where $b\in\{\tL,\ldots,\tN\}$,
then
\begin{multline}\label{eqn:LtoLR}
T_{(wb)-\tL}T_{(wb)_L}-T_{(wb)_L-\tL}T_{wb}
=\frac{p}{p-1}\Bigl(T_{w\tO}T_{(w\tO)_{LR}}-T_{(w\tO)_L}T_{(w\tO)_R}\Bigr).
\end{multline}
If moreover $w=w_{\mu-1}\cdots w_r\tO^r\in W$, where $r\geq 0$ is maximal,
and $w_L\neq\varepsilon$ is satisfied,
we have
\begin{equation}\label{eqn:LtoLR2}
x^{r+1}\Bigl(T_{w-\tL}T_{w_L}-T_{w_L-\tL}T_{w}\Bigr)
=\frac{1}{p-1}\Bigl(T_{w\tO}T_{(w\tO)_{LR}}-T_{(w\tO)_L}T_{(w\tO)_R}\Bigr).
\end{equation}
Let us now prove these formulas.
Write $w=w_{\mu-1}\cdots w_r\tO^r$ with $r\geq 0$ maximal.
We handle the case $w_L=\varepsilon$ separately.
In this case, we have
\begin{align*}
\hspace{1em}&\hspace{-1em}
T_{(wb)-\tL}T_{(wb)_L}-T_{(wb)_L-\tL}T_{wb}
\\&=
\bigl(bT_w+(p-b)x^{r+1}T_{w-1}\bigr)T_b
-T_{b-1}\bigl((b+1)T_w+(p-b-1)x^{r+1}T_{w-1}\bigr)
\\&=x^{r+1}\bigl((p-b)(b+1)-b(p-b-1)\bigr)T_{w-1}
\\&=x^{r+1}pT_{w-1}
\end{align*}
and
\begin{align*}
T_{w\tO}T_{(w\tO)_{LR}}-T_{(w\tO)_L}T_{(w\tO)_R}
&=\bigl(T_w+(p-1)x^{r+1}T_{w-1}\bigr)-T_w
\\&=x^{r+1}(p-1)T_{w-1},
\end{align*}
which yields~\eqref{eqn:LtoLR} for the case $w_L=\varepsilon$.
Assume now that $w_L\neq \varepsilon$.
Then $r$ is also the number of zeros at the low digits of $w_L$.
Therefore
\begin{align*}
\hspace{1em}&\hspace{-1em}
T_{(wb)-\tL}T_{(wb)_L}-T_{(wb)_L-\tL}T_{wb}
\\&=\bigl(bT_w+(p-b)x^{r+1}T_{w-\tL}\bigr)\bigl((b+1)T_{w_L}+(p-b-1)x^{r+1}T_{w_L-\tL}\bigr)\\
&-\bigl(bT_{w_L}+(p-b)x^{r+1}T_{w_L-\tL}\bigr)\bigl((b+1)T_w+(p-b-1)x^{r+1}T_{w-\tL}\bigr)\\
&=px^{r+1}\bigl(T_{w-\tL}T_{w_L}-T_{w_L-\tL}T_w\bigr),
\end{align*}
moreover
\begin{align*}
\hspace{1em}&\hspace{-1em}
T_{w\tO}T_{(w\tO)_{LR}}-T_{(w\tO)_L}T_{(w\tO)_R}
\\&=\bigl(T_w+(p-1)x^{r+1}T_{w-1}\bigr)T_{w_L}
+\bigl(T_{w_L}+(p-1)x^{r+1}T_{w_L-1}\bigr)T_{w}
\\&=(p-1)x^{r+1}\bigl(T_{w-1}T_{w_L}-T_{w_L-1}T_w\bigr),
\end{align*}
which proves the claim.

We have to treat the cases~\eqref{item:first_case}--\eqref{item:third_case}.
Assume that $\widetilde w=wb\tO$,
where $b\in\{\tL,\ldots,\tN\}$.
We have $\widetilde w_L=(wb)_L\tO$ and therefore
we obtain by~\eqref{eqn:LtoLR}
\begin{align*}
\hspace{1em}&\hspace{-1em}
T_{\widetilde w}T_{\widetilde w_{LR}}-T_{\widetilde w_L}T_{\widetilde w_R}
=
T_{wb\tO}T_{(wb)_L}-T_{(wb)_L\tO}T_{wb}
\\&=
\bigl(T_{wb}+(p-1)xT_{(wb)-\tL}\bigr)T_{(wb)_L}
-\bigl(T_{(wb)_L}+(p-1)xT_{(wb)_L-\tL}\bigr)T_{wb}\\
&=(p-1)x\bigl(T_{(wb)-\tL}T_{(wb)_L}-T_{(wb)_L-\tL}T_{wb}\bigr)\\
&=px
\bigl(T_{w\tO} T_{(w\tO)_{LR}}-T_{(w\tO)_L} T_{(w\tO)_R}\bigr).
\end{align*}
It follows that
\begin{align*}
\overline T_{\widetilde w}\overline T_{\widetilde w_{LR}}
-\overline T_{\widetilde w_L}\overline T_{\widetilde w_R}
&=\frac{px}{(b+1)^2}
\bigl(\overline T_{w\tO}\overline T_{(w\tO)_{LR}}-\overline T_{(w\tO)_L}\overline T_{(w\tO)_R}\bigr).
\end{align*}
Since the right depth of $w\tO$ is smaller than $d$,
we can apply the induction hypothesis. This finishes the case~\eqref{item:first_case}.
Now we assume that $\widetilde w=wb\tO^t$,
where $b\in\{\tL,\ldots,\tN\}$ and $t\geq 2$.
We first note that for a finite word $v\in \{\tO,\ldots,\tN\}^*$
we have the identity
$T_{vb\tO^t}=T_{vb\tO^{t-1}}+(p-1)x^tT_{vb\tO^{t-1}-\tL}
=T_{vb\tO^{t-1}}+(p-1)x^tp^{t-1}T_{v(b-\tL)}$,
analogously for $t-1$ instead of $t$,
therefore
\[T_{vb\tO^t}=(1+px)T_{vb\tO^{t-1}}-pxT_{vb\tO^{t-2}}.\]
We may therefore calculate:
\begin{align*}
T_{\widetilde w}T_{\widetilde w_{LR}}-T_{\widetilde w_L}T_{\widetilde w_R}
&=
\Bigl((1+px)T_{wb\tO^{t-1}}-pxT_{wb\tO^{t-2}}\Bigr)T_{w_Lb\tO^{t-1}}
\\&-
\Bigl((1+px)T_{w_Lb\tO^{t-1}}-pxT_{w_Lb\tO^{t-2}}\Bigr)T_{wb\tO^{t-1}}\\
&=
px\bigl(T_{wb\tO^{t-1}}T_{(wb\tO^{t-1})_{LR}}-T_{(wb\tO^{t-1})_L}T_{(wb\tO^{t-1})_R}\bigr).
\end{align*}
It follows that
\[
\overline T_{\widetilde w}\overline T_{\widetilde w_{LR}}-\overline T_{\widetilde w_L}\overline T_{\widetilde w_R}
=
px\bigl(\overline T_{wb\tO^{t-1}}\overline T_{(wb\tO^{t-1})_{LR}}-\overline T_{(wb\tO^{t-1})_L}\overline T_{(wb\tO^{t-1})_R}\bigr)
\]
and we can use the induction hypothesis.
We proceed to the third case.
Assume that $\widetilde w=wa$,
where $w=w_{\mu-1}\cdots w_r\tO^r$ and $r\geq 0$ is maximal, and $a\in\{\tL,\ldots,\tN-\tL\}$.
In the case that $w_L=\varepsilon$, we have
\begin{align*}
T_{\widetilde w}T_{\widetilde w_{LR}}-T_{\widetilde w_L}T_{\widetilde w_R}
&=T_{wa}-(a+1)T_w
\\&=(p-a-1)x^{r+1}T_{w-1}
\\&=\frac{p-a-1}{p-1}
\bigl(T_{w\tO}T_{(w\tO)_{LR}}-T_{(w\tO)_L}-T_{(w\tO)_R}\bigr).
\end{align*}

If $w_L\neq\varepsilon$, we obtain by~\eqref{eqn:LtoLR2}
\begin{align*}
T_{\widetilde w}T_{\widetilde w_{LR}}-T_{\widetilde w_L}T_{\widetilde w_R}
&= \bigl((a+1)T_w+(p-a-1)x^{r+1}T_{w-\tL}\bigr)T_{w_L}\\
&-\bigl((a+1)T_{w_L}+(p-a-1)x^{r+1}T_{w_L-\tL}\bigr)T_w
\\&=(p-a-1)x^{r+1}\bigl(T_{w-\tL}T_{w_L}-T_{w_L-\tL}T_w\bigr)\\
&=
\frac{p-a-1}{p-1}
\bigl(T_{w\tO}T_{(w\tO)_{LR}}-T_{(w\tO)_L}T_{(w\tO)_R}\bigr).
\end{align*}
therefore
\begin{align*}
\overline T_{\widetilde w}\overline T_{\widetilde w_{LR}}-\overline T_{\widetilde w_L}\overline T_{\widetilde w_R}&=
\frac{p-a-1}{a+1}\frac{1}{p-1}
\bigl(\overline T_{w\tO}\overline T_{(w\tO)_{LR}}-\overline T_{(w\tO)_L}\overline T_{(w\tO)_R}\bigr).
\end{align*}
Now one of the cases~\eqref{item:first_case} or~\eqref{item:second_case} is applicable and it is readily checked that~\eqref{eqn:r_numerator_equivalent} is satisfied.
The proof is complete.
\end{proof}

\begin{thebibliography}{10}

\bibitem{AB1997}
{\sc J.-P. Allouche and V.~Berth{\'e}}, {\em Triangle de {P}ascal, complexit\'e
  et automates}, Bull. Belg. Math. Soc. Simon Stevin, 4 (1997), pp.~1--23.
\newblock Journ{\'e}es Montoises (Mons, 1994).

\bibitem{AS1992}
{\sc J.-P. Allouche and J.~Shallit}, {\em The ring of {$k$}-regular sequences},
  Theoret. Comput. Sci., 98 (1992), pp.~163--197.

\bibitem{AS2003}
\leavevmode\vrule height 2pt depth -1.6pt width 23pt, {\em Automatic
  sequences}, Cambridge University Press, Cambridge, 2003.
\newblock Theory, applications, generalizations.

\bibitem{AS2008}
{\sc T.~Amdeberhan and R.~P. Stanley}, {\em {P}olynomial {C}oefficient
  {E}numeration},  (2008).
\newblock Preprint. arXiv:0811.3652.

\bibitem{BG2001}
{\sc G.~Barat and P.~J. Grabner}, {\em Distribution of binomial coefficients
  and digital functions}, J. London Math. Soc. (2), 64 (2001), pp.~523--547.

\bibitem{BG1996}
{\sc D.~Barbolosi and P.~J. Grabner}, {\em Distribution des coefficients
  multinomiaux et {$q$}-binomiaux modulo {$p$}}, Indag. Math. (N.S.), 7 (1996),
  pp.~129--135.

\bibitem{C1967}
{\sc L.~Carlitz}, {\em The number of binomial coefficients divisible by a fixed
  power of a prime}, Rend. Circ. Mat. Palermo (2), 16 (1967), pp.~299--320.

\bibitem{C1992}
{\sc E.~Cateland}, {\em {Digital sequences and k-regular sequences}}, thesis,
  {Universit{\'e} Sciences et Technologies - Bordeaux I}, June 1992.

\bibitem{DW1990}
{\sc K.~S. Davis and W.~A. Webb}, {\em Lucas' theorem for prime powers},
  European J. Combin., 11 (1990), pp.~229--233.

\bibitem{D1919}
{\sc L.~E. {Dickson}}, {\em {History of the theory of numbers. Vol. I:
  Divisibility and primality.}}, 1919.
\newblock Chapter IX: ``Divisibility of factorials and multinomial
  coefficients''.

\bibitem{DKS2016}
{\sc M.~Drmota, M.~Kauers, and L.~Spiegelhofer}, {\em On a {C}onjecture of
  {C}usick {C}oncerning the {S}um of {D}igits of {$n$} and {$n+t$}}, SIAM J.
  Discrete Math., 30 (2016), pp.~621--649.
\newblock arXiv:1509.08623.

\bibitem{O}
{\sc J.~W. Eaton, D.~Bateman, and S.~Hauberg}, {\em {GNU Octave} version 3.0.1
  manual: a high-level interactive language for numerical computations},
  CreateSpace Independent Publishing Platform, 2009.
\newblock {ISBN} 1441413006.

\bibitem{EH2016}
{\sc J.~Emme and P.~Hubert}, {\em {C}entral {L}imit {T}heorem for probability
  measures defined by sum-of-digits function in base 2},  (2016).
\newblock Preprint. arXiv:1605.06297.

\bibitem{EP2015}
{\sc J.~Emme and A.~Prikhodko}, {\em On the asymptotic behaviour of the
  correlation measure of sum-of-digits function in base $2$},  (2015).
\newblock Preprint. arXiv:1504.01701.

\bibitem{F1947}
{\sc N.~J. Fine}, {\em Binomial coefficients modulo a prime}, Amer. Math.
  Monthly, 54 (1947), pp.~589--592.

\bibitem{FGKPT1994}
{\sc P.~Flajolet, P.~Grabner, P.~Kirschenhofer, H.~Prodinger, and R.~F. Tichy},
  {\em Mellin transforms and asymptotics: digital sums}, Theoret. Comput. Sci.,
  123 (1994), pp.~291--314.

\bibitem{FS2009}
{\sc P.~Flajolet and R.~Sedgewick}, {\em Analytic combinatorics}, Cambridge
  University Press, Cambridge, 2009.

\bibitem{F1967}
{\sc R.~D. Fray}, {\em Congruence properties of ordinary and {$q$}-binomial
  coefficients}, Duke Math. J., 34 (1967), pp.~467--480.

\bibitem{G1899}
{\sc J.~Glaisher}, {\em On the residue of a binomial-theorem coefficient with
  respect to a prime modulus}, Quarterly Journal of Pure and Applied
  Mathematics, 30 (1899), pp.~150--156.

\bibitem{G1992}
{\sc A.~Granville}, {\em Zaphod {B}eeblebrox's brain and the fifty-ninth row of
  {P}ascal's triangle}, Amer. Math. Monthly, 99 (1992), pp.~318--331.

\bibitem{G1997}
\leavevmode\vrule height 2pt depth -1.6pt width 23pt, {\em Arithmetic
  properties of binomial coefficients. {I}. {B}inomial coefficients modulo
  prime powers}, vol.~20 of CMS Conf. Proc., Amer. Math. Soc., Providence, RI,
  1997.

\bibitem{H1970}
{\sc F.~T. Howard}, {\em A combinatorial problem and congruences for the
  {R}ayleigh function}, Proc. Amer. Math. Soc., 26 (1970), pp.~574--578.

\bibitem{H1971}
\leavevmode\vrule height 2pt depth -1.6pt width 23pt, {\em The number of
  binomial coefficients divisible by a fixed power of {$2$}}, Proc. Amer. Math.
  Soc., 29 (1971), pp.~236--242.

\bibitem{H1973}
\leavevmode\vrule height 2pt depth -1.6pt width 23pt, {\em Formulas for the
  number of binomial coefficients divisible by a fixed power of a prime}, Proc.
  Amer. Math. Soc., 37 (1973), pp.~358--362.

\bibitem{HSW1997b}
{\sc J.~G. Huard, B.~K. Spearman, and K.~S. Williams}, {\em On {P}ascal's
  triangle modulo {$p^2$}}, Colloq. Math., 74 (1997), pp.~157--165.

\bibitem{K1968}
{\sc G.~S. Kazandzidis}, {\em Congruences on the binomial coefficients}, Bull.
  Soc. Math. Gr\`ece (N.S.), 9 (1968), pp.~1--12.

\bibitem{KW1989}
{\sc D.~E. Knuth and H.~S. Wilf}, {\em The power of a prime that divides a
  generalized binomial coefficient}, J. Reine Angew. Math., 396 (1989),
  pp.~212--219.

\bibitem{K1852}
{\sc E.~E. Kummer}, {\em {\"Uber die Erg\"anzungss\"atze zu den allgemeinen
  Reciprocit\"atsgesetzen}}, J. Reine Angew. Math., 44 (1852), pp.~93--146.

\bibitem{L1878}
{\sc E.~Lucas}, {\em Sur les congruences des nombres eul\'eriens et les
  coefficients diff\'erentiels des functions trigonom\'etriques suivant un
  module premier}, Bull. Soc. Math. France, 6 (1878), pp.~49--54.

\bibitem{Sloane}
{\sc {N. J. A. Sloane}}, {\em {T}he {O}n-{L}ine {E}ncyclopedia of {I}nteger
  {S}equences}.
\newblock published electronically at {\tt https://oeis.org}.

\bibitem{P1914}
{\sc O.~{Perron}}, {\em {\"Uber das infinit\"are Verhalten der Koeffizienten
  einer gewissen Potenzreihe.}}, {Arch. der Math. u. Phys. (3)}, 22 (1914),
  pp.~329--340.

\bibitem{R2011}
{\sc E.~Rowland}, {\em The number of nonzero binomial coefficients modulo
  {$p^\alpha$}}, J. Comb. Number Theory, 3 (2011), pp.~15--25.

\bibitem{R2017}
\leavevmode\vrule height 2pt depth -1.6pt width 23pt, {\em {A} matrix
  generalization of a theorem of {F}ine},  (2017).
\newblock Preprint. arXiv:1704.05872.

\bibitem{S1974a}
{\sc D.~Singmaster}, {\em Notes on binomial coefficients. {I}. {A}
  generalization of {L}ucas' congruence}, J. London Math. Soc. (2), 8 (1974),
  pp.~545--548.

\bibitem{S1974c}
\leavevmode\vrule height 2pt depth -1.6pt width 23pt, {\em Notes on binomial
  coefficients. {III}. {A}ny integer divides almost all binomial coefficients},
  J. London Math. Soc. (2), 8 (1974), pp.~555--560.

\bibitem{Si1980}
\leavevmode\vrule height 2pt depth -1.6pt width 23pt, {\em Divisibility of
  binomial and multinomial coefficients by primes and prime powers}, Fibonacci
  Assoc., Santa Clara, Calif., 1980.

\bibitem{SW1999}
{\sc B.~K. Spearman and K.~S. Williams}, {\em On a formula of {H}oward}, Bull.
  Hong Kong Math. Soc., 2 (1999), pp.~325--340.
\newblock (Available on {S}pearman's website).

\bibitem{S1977}
{\sc K.~B. Stolarsky}, {\em Power and exponential sums of digital sums related
  to binomial coefficient parity}, SIAM J. Appl. Math., 32 (1977),
  pp.~717--730.

\bibitem{S}
{\sc {The Sage Developers}}, {\em {S}ageMath, the {S}age {M}athematics
  {S}oftware {S}ystem ({V}ersion 6.9)}, 2015.
\newblock {\tt http://www.sagemath.org}.

\bibitem{W1990}
{\sc W.~A. Webb}, {\em The number of binomial coefficients in residue classes
  modulo {$p$} and {$p^2$}}, Colloq. Math., 60/61 (1990), pp.~275--280.

\bibitem{W1933}
{\sc E.~M. Wright}, {\em The {C}oefficients of a {C}ertain {P}ower {S}eries},
  J. London Math. Soc., S1-7 (1933), p.~256.

\end{thebibliography}

\end{document}